\documentclass[final, a4paper, 12pt]{article}

\usepackage[a4paper,width=15cm]{geometry}
\usepackage{amsfonts}
\usepackage{amsmath}
\usepackage{amssymb}
\usepackage{amscd}
\usepackage{amsthm}
\usepackage{listings}
\usepackage{setspace}
\usepackage{bbm}
\usepackage{bm}
\usepackage{cite}
\usepackage{showkeys}
\usepackage[english]{babel}
\usepackage{dblaccnt}
\usepackage{accents}
\usepackage{graphicx}
\usepackage{hyperref}

\usepackage[mathscr]{euscript}
\usepackage{color}
\usepackage{fancyhdr}
\usepackage{dsfont}

\newtheorem{definition}{Definition}
\newtheorem{lemma}[definition]{Lemma}
\newtheorem{theorem}[definition]{Theorem}
\newtheorem{corollary}[definition]{Corollary}

\newtheorem{remark}[definition]{Remark}

\newcommand*{\N}{\ensuremath{\mathbb{N}}}
\newcommand*{\Z}{\ensuremath{\mathbb{Z}}}
\newcommand*{\R}{\ensuremath{\mathbb{R}}}
\newcommand*{\C}{\ensuremath{\mathbb{C}}}
\renewcommand{\i}{\mathrm{i}}
\renewcommand{\phi}{\varphi}
\renewcommand{\d}[1]{\,\mathrm{d}#1 \,}
\newcommand{\dS}[1]{\,\mathrm{dS}#1 \,}

\newcommand{\1}{\mathds{1}}
\renewcommand{\S}{\mathbb{S}}
\renewcommand{\Re}{\mathrm{Re}\,}
\renewcommand{\Im}{\mathrm{Im}\,}

\renewcommand{\div}{\mathrm{div}}

\newcommand{\cmp}{{\mathrm{cmp}}}

\DeclareMathOperator{\supp}{\mathrm{supp}}

\renewcommand{\rho}{{\varrho}}
\renewcommand{\epsilon}{{\varepsilon}}
\newcommand{\loc}{{\mathrm{loc}}}
\newcommand{\out}{{\mathrm{out}}}
\newcommand{\inn}{{\mathrm{in}}}

\newcommand{\marker}[1]{{#1}} 

\begin{document}

\title{\marker{A Factorization Method and Monotonicity Bounds in Inverse Medium Scattering for Contrasts with Fixed Sign on the Boundary}}
\author{Evgeny Lakshtanov\thanks{Department of Mathematics, University of Aveiro, 3810-193 Aveiro, Portugal; this work was supported by Portuguese funds through the CIDMA - Center for Research and Development in Mathematics and Applications, and the Portuguese Foundation for Science and Technology (``FCT--Fund\c{c}\~{a}o para a Ci\^{e}ncia e a Tecnologia''), within project UID/MAT/04106/2013 \texttt{lakshtanov@ua.pt}} \and
Armin Lechleiter\thanks{Center for Industrial Mathematics, University of Bremen, 28359 Bremen, Germany; \texttt{lechleiter@math.uni-bremen.de}; this work was supported through an exploratory project granted by the University of Bremen in the framework of its institutional strategy, funded by the excellence initiative of the federal and state governments of Germany.}}

\maketitle

\begin{abstract}
We generalize the factorization method for inverse medium scattering using a particular factorization of the difference of two far field operators.
Whilst the factorization method been used so far mainly to identify the shape of a scatterer's support, we show that factorizations based on Dirichlet-to-Neumann operators can be used to compute bounds for numerical values of the medium on the boundary of its support.
To this end, we generalize ideas from inside-outside duality to obtain a monotonicity principle that allows for alternative uniqueness proofs for particular inverse scattering problems (e.g., when obstacles are present inside the medium).
This monotonicity principle indeed is our most important technical tool:
It further directly shows that the boundary values of the medium's contrast function are uniquely determined by the corresponding far field operator.
Our particular factorization of far field operators additionally implies that the factorization method rigorously characterizes the support of an inhomogeneous medium if the contrast function takes merely positive or negative values on the boundary of its support, independent of the contrast's values inside its support.
Finally, the monotonicity principle yields a simple algorithm to compute upper and lower bounds for these boundary values, assuming the support of the contrast is known.
Numerical experiments show feasibility of a resulting numerical algorithm.
\end{abstract}

\section{Introduction}
The factorization method is well-known to identify the shape of scattering objects from measurements of near or far field data for various models of time-harmonic wave propagation~\cite{Kirsc2008}.
It is notably able to detect regions where known inhomogeneous media are perturbed by either changes in the wave speed, in the density, or by obstacles~\cite{nachman2007imaging,Cakon2015}.
In particular in the latter case, classical uniqueness proofs in inverse scattering theory based on Calderon's property of completeness of products of solutions typically fail.
The method's flexibility with respect to the model however faces a crucial positivity assumption on the middle operator in the data operator's factorization that gives the method its name.
Additionally, it seems complicated to extend the method towards reconstructing information on numerical values of material parameters.
(See~\cite{Kirsch2011} for such an attempt in impedance tomography.)

\marker{
In this paper, we use a factorization of the far field operator for a smooth, scalar and real-valued contrast (i.e., an isotropic non-absorbing inhomogeneous medium) from~\cite{Laksh2013b} in function spaces on the boundary of the scatterer to obtain a sign-definite factorization if the contrast function is, roughly speaking, strictly positive or strictly negative on the boundary of the scatterer.
This factorization firstly implies that the factorization method is rigorously applicable to inhomogeneous media if the smooth, real-valued contrast takes strictly positive or strictly negative boundary values, independent of the values the contrast takes inside its support.
Secondly, we deduce a uniqueness theorem for the values of contrast on the boundary of its support given far field data of the scattering object, and thirdly we obtain a simple monotonicity-type algorithm computing upper and lower bounds for these boundary values, which is briefly sketched and demonstrated via numerical examples.
Further consequences include for instance uniqueness results for scattering problems involving obstacles inside inhomogeneous media.%
}

Our approach can be roughly described as follows:
We compare a measured far field operator $F_1$ corresponding to \marker{an unknown, real-valued} contrast $q_1$ with an auxiliary \marker{far field} operator $F_2$ corresponding to a second \marker{artificial, real-valued} contrast $q_2$.
Writing $\mathcal{S}_2$ for the scattering operator for $q_2$, it is easy to show that operator $\mathcal{S}_2^\ast(F_1-F_2)$ is normal.
We further show that the real part of its quadratic form is sign-definite if $q_1 - q_2 \gtreqless 0$ in $\R^d$.
Via techniques from pseudo-differential operator theory we refine this result by demonstrating that this form is, roughly speaking, sign-definite if and only if  $q_1 - q_2 \gtrless 0$ on the boundary of the common support $D$ of $q_{1,2}$.
This is one of the few monotonicity results in scattering theory: If $q_1 > q_2$ (or $q_1 < q_2$) on $\partial D$, then the real part of the quadratic form of $\mathcal{S}_2^\ast(F_1-F_2)$ is negative (positive), up to a finite-dimensional perturbation.
It is based on a factorization of $F_{1,2}$ via Dirichlet-to-Neumann operators from~\cite{Laksh2013b}.

The rest of this paper is structured as follows:
We briefly review theory on the direct scattering problem in Section~\ref{se:direct} and show in Section~\ref{se:oldStyle} that the real parts of the eigenvalues of $\mathcal S_2^\ast (F_1 - F_2)$ relate to the sign of $q_1 - q_2$ in $\R^d$.
Section~\ref{se:facDtN} then characterizes the sign of all but finitely many real parts of these eigenvalues by the sign of $q_1 - q_2$ on the boundary of their joint support.
Finally, Section~\ref{se:appl} treats several applications of this result, providing algorithms for particular inverse scattering problems.

\section{The forward scattering problem}\label{se:direct}
Consider a wave number $k>0$, a real-valued contrast function $q: \R^d \to \R$, and an entire solution $u^i$ of the Helmholtz equation $\Delta u^i + k^2 \, u^i = 0$ in $\R^d$.
The forward scattering problem then seeks for a total field $u$ solving
\begin{equation}\label{eq:helmholtz}
  \Delta u + k^2 (1+q) u = 0 \qquad \text{in } \R^d,
\end{equation}
subject to Sommerfeld's radiation condition for the scattered field $u^s  = u - u^i$,
\begin{equation}\label{eq:SRC}
  \lim_{r\to\infty} r^{(m-1)/2} \left( \frac{\partial u^s}{\partial r}(r \hat{x}) - \i k u^s(r \hat{x}) \right)
  = 0,
  \qquad |\hat{x}| =1,
\end{equation}
uniformly in all $\hat{x} \in \S^{d-1}=\{ x\in\R^d, \, |x|=1\}$.
The scattering problem~(\ref{eq:helmholtz}-\ref{eq:SRC}) possesses a unique weak solution $u\in H^2_\loc(\R^d)$ if, e.g., $q\in L^\infty(\R^d,\C)$ satisfies $\Im(q) \geq 0$, see~\cite{Colto2013}.
Under these assumptions, the evaluation of the far field $u^\infty=u^\infty_q:\, \S \to \C$ of the scattered field $u^s$ at the point $\hat{x} \in \S$ is defined by
\[
  u^s(r\hat{x}) = \gamma_d \frac{\exp(\i k r)}{r} u^\infty(\hat{x}) + \mathcal{O}\left(\frac1{r^2}\right) \quad \text{as } r \to \infty,
  \qquad
  \gamma_d = \begin{cases} \frac1{4\pi} & d=3, \\ \frac{\exp(\i \pi/4)}{\sqrt{8\pi k}} & d=2, \end{cases}
\]
and possesses for each $R>0$ with $\supp(q) \Subset B_R$ the representation
\begin{equation} \label{eq:farFieldBall}
  u^\infty(\hat{x})
  = \int_{\partial B_R} \left[ u^s(y) \frac{\partial e^{-\i k\, y \cdot \hat{x}}}{\partial\nu(y)}  - \frac{\partial u^s(y)}{\partial\nu(y)} e^{-\i k\, y \cdot \hat{x}} \right] \dS{(y)}, \qquad \hat{x} \in \S^{d-1},
\end{equation}
where $\nu$ here and elsewhere denotes the outer unit normal to $D$.
For incident plane waves $u^i(x,\theta) = \exp(\i k \, x\cdot \theta)$ of direction $\theta \in \S$ we denote from now on the dependence of $u = u(\cdot,\theta)$, $u^s = u^s(\cdot, \theta)$, and $u^\infty = u^\infty(\cdot, \theta)$ on the  incident direction $\theta$ explicitly.
The far field pattern $(\hat{x}, \theta) \mapsto u^\infty(\hat{x}, \theta)$ then defines the far field operator
\begin{equation}\label{eq:Fq}
  F=F_q: \, L^2(\S)\to L^2(\S),
  \qquad
  g \mapsto F g(\hat{x}) = \int_\S u^\infty(\hat{x},\theta) g(\theta) \dS{(\theta)}.
\end{equation}
We recall that the far field operator is normal if the contrast $q$ has compact support and is real-valued, see~\cite{Colto2013}.
For simplicity we denote this set of functions by
\[
  L^\infty_\cmp(\R^d,\R) = \left\{ q \in L^\infty(\R^d), \, q \text{ is real-valued and }  \supp(q) \text{ is compact} \right\}
\]
and assume that all contrasts considered in the sequel belong to this set.
We further define the scattering operator
\[
  \mathcal{S} = \mathcal{S}_q:\, L^2(\S) \to L^2(\S),
  \qquad
  \mathcal{S} = I + 2\i k |\gamma_d|^2 \, F_q.
\]
\begin{lemma}\label{eq:diffNormal}
  If $q_{1,2}\in L^\infty_\cmp(\R^d,\R)$ with associated far field- and scattering operators $F_{1,2}$ and $\mathcal{S}_{1,2}$, then $\mathcal{S}_2 ^\ast (F_1 - F_2)$ is a normal operator on $L^2(\S)$.
\end{lemma}
\begin{proof}
For any far field operator with real-valued contrast, the corresponding scattering operator is unitary.
Thus,
\[
  \mathcal{S}_2^\ast (F_1 - F_2)
  = \frac{1}{2\i k |\gamma_d|^2} \mathcal{S}_2^\ast \left( \mathcal{S}_1 - \mathcal{S}_2 \right)
  = \frac{1}{2\i k |\gamma_d|^2} \left( \mathcal{S}_2^\ast \mathcal{S}_1 - I \right).
\]
As $\mathcal{S}_2^\ast \mathcal{S}_1$ is normal (since $\mathcal{S}_{1,2}$ is unitary), the operator $\mathcal{S}_2^\ast (F_1 - F_2)$ is normal, too.
\end{proof}

\section{Factorization via Herglotz operators}\label{se:oldStyle}
We prove in this section a factorization of $\mathcal{S}_2^\ast (F_1 - F_2)$ using Herglotz operators which shows that the real parts of the eigenvalues of that operator are sign-definite if, roughly speaking, $q_1-q_2$ is either greater or less than zero on $\supp(q_1-q_2)$.
For scattering from a penetrable medium modeled by the differential equation $\div (A \nabla u) + k^2 (1+q) u = 0$ and additionally containing an inclusion, a related factorization can be found in~\cite[Th.~3.1 \& Th~4.7]{Cakon2015}.
We formulate this lemma using two contrasts $q_{1,2}$ as parameters in the Helmholtz equation~\eqref{eq:helmholtz} and denote the corresponding total, scattered, and far fields for incident plane waves \marker{of direction $\theta \in \S^{d-1}$ by $u_{1,2}(\cdot,\theta)$, $u^s_{1,2}(\cdot,\theta)$, and $u^\infty_{1,2}(\cdot,\theta)$}, as well as the corresponding far field and scattering operators by $F_{1,2}$ and $\mathcal{S}_{1,2}$, respectively.

\begin{lemma}\label{th:diffSign}
  If $q_{1,2}\in L^\infty_\cmp(\R^d)$, then $\mathcal{S}_2^\ast (F_1 - F_2) = H_2^\ast T_{1\& 2} H_2$, where the operator $H_2: \, L^2(\S^{d-1}) \to L^2(\supp(q_1-q_2))$ is defined by
  \begin{equation}\label{eq:herglotzOperator}
    g \mapsto \left. v_g \right|_{\supp(q_1-q_2)},
    \qquad v_g = \int_\S u_2(\cdot,\theta) g(\theta) \dS{(\theta)},
  \end{equation}
  and $T_{1\& 2}$ is defined on $L^2(\supp(q_1-q_2))$ by $T_{1\& 2} f = k^2 (q_1-q_2) \big(f+v|_{\supp(q_1-q_2)} \big)$, where $v \in H^1_\loc(\R^d)$ is the weak, radiating solution to
  \begin{equation}\label{eq:T}
    \Delta v + k^2 (1+q_1) v = - k^2 (q_1-q_2) f
    \qquad \text{in } \R^d.
  \end{equation}
  \marker{Both $H_2$ and $T_{1\& 2}$ are continuous and $H_2$ is compact and injective; if $q_{1,2}\in L^\infty_\cmp(\R^d, \R)$ are real-valued, then $\Im T_{1\& 2} \geq 0$, and $q_1\not = q_2$ in $L^2(\supp(q_1-q_2))$ implies that $T_{1\& 2}$ is injective.}
\end{lemma}
\begin{proof}
(1) Set $D=\supp(q_1-q_2)$, denote by \marker{$v_g^{(2)}=v_g$ the function from~\eqref{eq:herglotzOperator} for some $g\in L^2(\S)$, by $v_g^{(1)} =\int_\S u_1(\cdot,\theta)g(\theta) \dS{(\theta)}$, and by $v_g^{(1,2),s}$} the corresponding two scattered fields for $q_{1,2}$. 
Note that \marker{$v_g^{(1,2)}$} hence solves the differential equation \marker{$\Delta v_g^{(1,2)} + k^2 (1+q_{1,2}) v_g^{(1,2)} = 0$} in $\R^d$.
The difference $\marker{\marker{\widetilde{v}} = v_g^{(1),s} - v_g^{(2),s}} \in H^1_\loc(\R^d)$ is the unique radiating solution to
\begin{equation}\label{eq:aux160}
  \Delta \marker{\widetilde{v}} + k^2 (1+q_1) \marker{\widetilde{v}} = - k^2 (q_1-q_2) v_g^{(2)}
  \qquad \text{in } \R^d.
\end{equation}
This motivates to define $G: \, L^2(D) \to L^2(\S)$ by $G f = \marker{\widetilde{v}}^\infty$, where $\marker{\widetilde{v}} \in H^1_\loc(\R^d)$ is the radiating solution to~\eqref{eq:aux160} with \marker{$v_g^{(2)}$} on the right replaced by $f$ (extended by zero to all of $\R^d$).
Consequently, the definition of $H_2$ in~\eqref{eq:herglotzOperator} shows that $F_1 - F_2 = GH_2$.

(2) To obtain the indicated factorization of $\mathcal{S}_2^\ast (F_1 - F_2)$ we rely on
the weak, radiating solution $w \in H^1_\loc(\R^d)$ to
\begin{equation}\label{eq:middleOperator}
    \Delta w + k^2 (1+q_2) w = - f
    \qquad \text{in } \R^d,
\end{equation}
\marker{as well as on the exterior Dirichlet-to-Neumann operator $\Lambda$ for radiating solutions to the Helmholtz equation $\Delta w + k^2 w = 0$ in the exterior of the ball $B_R$, see~\cite{Colto2013}.} 
A partial integration in $B_R$ and the far field representation~\eqref{eq:farFieldBall} show that
\begin{align*}
  (f, H_2 g)_{L^2(D)}
  & = \int_{B_R} \left[ \nabla w \cdot \nabla \overline{v_g} - k^2 (1+q_2) w \overline{v}_g \right] \d{x} - \int_{\partial B_R} \Lambda \big(w|_{\partial B_R}\big) \overline{v_g} \dS{} \\
  & = - \int_{B_R} w \left[ \Delta \overline{v_g} + k^2 (1+q_2) \overline{v}_g \right] \d{x} - \int_{\partial B_R} \left[ \frac{\partial w}{\partial\nu} \overline{v_g} - w \frac{\partial \overline{v_g}}{\partial\nu} \right] \dS{} \\
  & \stackrel{\eqref{eq:herglotzOperator}}{=}- \int_{\partial B_R} \bigg[ \frac{\partial w(y)}{\partial\nu} \int_\S \left( e^{-\i k \, y\cdot \theta}+u_2^s(y,\theta) \right) \overline{g(\theta)} \dS{(\theta)}     \\
  & \qquad \qquad - w(y) \frac{\partial}{\partial\nu(y)} \int_\S \left( e^{-\i k \, y\cdot \theta}+u_2^s(y,\theta) \right) \overline{g(\theta)} \dS{(\theta)} \bigg] \dS{(y)} \\
  & \stackrel{R\to\infty}{\longrightarrow} \int_{\S} w^\infty(\theta) \overline{g(\theta)}  \dS{(\theta)} - 2\i k |\gamma_d|^2 \int_\S w^\infty(\theta) \, \overline{F_2 g(\theta)} \dS{(\theta)},
\end{align*}
where the last term follows by the radiation condition~\eqref{eq:SRC} for the radiating function $w$.
Thus, $H_2^\ast f = w^\infty - 2\i k |\gamma_d|^2 \, F_2^\ast w^\infty = \mathcal{S}_2^\ast w^\infty$ and $\mathcal{S}_2 H_2^\ast f = w^\infty$.

(3) Rephrasing the Helmholtz equation~\eqref{eq:aux160} for $\marker{\widetilde{v}} \in H^1_\loc(\R^d)$ as $\Delta \marker{\widetilde{v}} + k^2 (1+q_2) \marker{\widetilde{v}} = - k^2 (q_1-q_2) (\marker{v_g^{(2)}} + \marker{\widetilde{v}})$ shows that the radiating solution $w$ to~\eqref{eq:middleOperator} with right-hand side $f$ replaced by $-k^2(q_1-q_2)(\marker{v_g^{(2)}} + \marker{\widetilde{v}})$ equals $\marker{\widetilde{v}}$.
Due to part (2) of the proof, we conclude that $\mathcal{S}_2 H_2^\ast  \big( k^2 (q_1-q_2) (\marker{v_g^{(2)}} + \marker{\widetilde{v}}) \big) = \marker{\widetilde{v}}^\infty$.
By~\eqref{eq:T}, there holds that $T_{1\& 2} (\marker{v_g^{(2)}}|_D) =  k^2 (q_1-q_2) (\marker{v_g^{(2)}} + \marker{\widetilde{v}})$ in $L^2(D)$ where $\overline{D}=\supp(q_1-q_2)$, such that
\[
  \mathcal{S}_2 H_2^\ast T_{1\& 2} \big( \marker{v_g^{(2)}}|_{D} \big)
  = \marker{\widetilde{v}}^\infty
  = G \big(\marker{v_g^{(2)}}|_{D} \big)
  \qquad \text{in } L^2(\S). 
\]
As $\marker{v_g^{(2)}}|_{D}= H_2 g$, we conclude that $\mathcal{S}_2 H_2^\ast T_{1\& 2} H_2 g = G H_2 g = (F_1 - F_2)g$.

(4) Continuity of $H_2$ and $T_{1\& 2}$ is clear, as well as the compactness of $H_2$ due to the smoothness of $u_2$.
Injectivity of $H_2$ follows from a unique continuation argument as in the classical case when $q_1$ vanishes.
For $T_{1\& 2}$, injectivity requires that $q_1\not = q_2$, since $T_{1\& 2} f = k^2 (q_1-q_2) (f+v) = 0$ is equivalent to $f=\marker{-}v$ on $\supp(q_1-q_2)$.
The differential equation~\eqref{eq:T} then shows that $v$ is the radiating solution to $\Delta v + k^2 (1+2q_1-q_2) v = 0$ in $\R^d$, such that $v$ must vanish entirely \marker{as $2q_1-q_2$ is real-valued}.

To show that $\Im T_{1\& 2} \geq 0$, we choose $f\in L^2(D) = L^2(\supp(q_1-q_2))$ and extend this function by zero to all of $\R^d$.
Recall that $T_{1\& 2} f = k^2 (q_1-q_2) (f+v|_{D})$, where $v \in H^1_\loc(\R^d)$ is the radiating solution to~\eqref{eq:T}.
Thus, abbreviating the scalar product of $L^2(D)$ by $(\cdot,\cdot)$,
\begin{align*}
  \Im (T_{1\& 2} f, f)
  &= k^2 \Im ((q_1-q_2) (f+v), (f+v)) - k^2 \Im ((q_1-q_2) (f+v), v)\\
  &= k^2 \Im ( (q_1-q_2)v, (f+v) ).
\end{align*}
since $q_{1,2}$ are both real-valued.
We reformulate the equation for $v$ as $\Delta v + k^2 (1+q_2) v = - k^2 (q_1-q_2) (f+v)$ in $\R^d$ and conclude by partial integration that
\begin{align}
  k^2 \Im( (q_1-q_2) v ,  \, (f+v) )
  &= k^2 \Im \int_{D} (q_1-q_2) v  \, (\overline{f}+\overline{v}) \d{x} \label{eq:imagT}\\
  & = \Im \int_{B_R}  v \left[ \Delta \overline{v} + k^2 (1+q_2) \overline{v} \right]  \d{x}
  =\Im \int_{\partial B_R} \frac{\partial \overline{v}}{\partial \nu} v  \d{S}.  \nonumber
\end{align}
The radiation condition~\eqref{eq:SRC} implies that $\int_{\partial B_R} (\partial \overline{v}/\partial \nu) v \d{S} \stackrel{R\to\infty}{\longrightarrow} (\i k |\gamma_d|^2) \int_\S |v^\infty|^2 \d{S}$, such that $\Im (T_{1\& 2} f, f)_{L^2(D)} \to k |\gamma_d|^2 \| v^\infty \|_{L^2(\S^{d-1})}^2 \geq 0$.
\end{proof}

Due to normality and compactness of $\mathcal{S}_2 ^\ast (F_1 - F_2)$, this operator possesses eigenvalues $\lambda_j = \lambda_j(q_1, q_2)$ and a complete orthonormal system of eigenvectors $\psi_j = \psi_j(q_1,q_2)$ in $L^2(\S)$, such that
\[
  \mathcal{S}_2 ^\ast (F_1 - F_2) g
  = \sum_{j\in\N} \lambda_j (g,\psi_j)_{L^2(\S)} \psi_j
  \qquad
  \text{for all } g \in L^2(\S).
\]

\begin{lemma}\label{th:eigsSign}
  (a) If $q_{1,2}\in L^\infty_\cmp(\R^d,\R)$ are two real-valued contrasts such that $q_1\geq q_2$ in $\R^d$ and $q_1- q_2 \geq c_0>0$ in $\supp(q_1-q_2)$, then $\Re \lambda_j(q_1,q_2) \geq 0$ for all but a finite number of $j\in\N$.
  If $q_1 \leq q_2$ in $\R^d$ and $q_2 - q_1\leq c_0>0$ in $\supp(q_1-q_2)$, then $\Re \lambda_j(q_1,q_2) \leq 0$ for all but a finite number of $j\in\N$. \\
  (b)  Under the assumptions of (a), the sequence of eigenvalues  $\lambda_j(q_1,q_2)$ belongs to the open first quadrant $Q_+ = \{ \Re \xi>0,\, \Im \xi >0 \} \cup \{ 0 \}$ of the complex plane joint with zero if $q_1 \geq q_2$ and $j$ is large enough.
If $q_1 \leq q_2$, the eigenvalues belong to the second quadrant $Q_- = \{ \Re \xi<0,\, \Im \xi >0 \} \cup \{ 0 \}$ of the complex plane joint with zero, if $j$ is large enough.
%
%
\end{lemma}
\begin{proof}
(a)
Assume for a moment that we have already proven that $\Re T_{1\& 2} = T_0 + K$ equals a \marker{self-adjoint} positive (or negative) definite operator $T_0$ plus a compact \marker{self-adjoint} perturbation $K$ if $q_1 \geq q_2$ in $\R^d$ (or $q_1 \leq q_2$ in $\R^d$).
As the arguments for negative definite $T_0$ are analogous to those for positive $T_0$, we merely consider positive definite $T_0$ from now on and abbreviate $D := \supp(q_1-q_2)$.
The factorization $\mathcal{S}_2^\ast (F_1-F_2) = H_2^\ast T_{1\& 2} H_2$ then implies that
\begin{align}
  \Re \big(\mathcal{S}_2^\ast (F_1-F_2) g,g \big)_{L^2(\S)}
  & = \Re \big(T_0 H_2 g, H_2g \big)_{L^2(D)} + \Re \big(K H_2 g, H_2g \big)_{L^2(D)} \nonumber\\
  & = \marker{\big(T_0 H_2 g, H_2g \big)_{L^2(D)} + \big(K H_2 g, H_2g \big)_{L^2(D)}} \nonumber\\ 
  & \geq c_0 \| H_2g \|_{L^2(D)}^2 + \Re \big(K H_2 g, H_2 g \big)_{L^2(D)}.\label{eq:operatorProductLowBound}
\end{align}
Plugging in the eigenvectors $\psi_j$ for $g$ and dividing by $\| H_2 \psi_j \|^2_{L^2(D)}$ hence yields that
\begin{equation}\label{eq:realBound}
  \frac{\Re \lambda_j}{\| H_2 \psi_j \|_{L^2(D)}^2} \geq c_0  + \bigg(K \frac{H_2 \psi_j}{\| H_2 \psi_j \|_{L^2(D)}}, \frac{H_2 \psi_j}{\| H_2 \psi_j \|_{L^2(D)}} \bigg)_{L^2(D)},
  \qquad j\in\N .
\end{equation}
If an infinite number of eigenvalues \marker{$\lambda_j$} has negative real part, $-K$ would be positive on an infinite-dimensional subspace, which is impossible by compactness of $K$.

We still need to show that $\Re T_{1\& 2} = T_0 + K$ is sum of a \marker{self-adjoint} positive definite operator $T_0$ plus a compact \marker{self-adjoint} perturbation $K$.
As in part (4) of the proof of Lemma~\ref{th:diffSign},
\begin{equation}\label{eq:aux252}
  \Re \big(T_{1\& 2} f, h \big)_{L^2(D)}
  = k^2 \int_D (q_1-q_2) f \, \overline{h} \d{x} + k^2 \Re ((q_1-q_2) v, h)_{L^2(\marker{B_R})}
\end{equation}
for $f,h \in L^2(D)$ extended by zero to all of $\R^d$, $v \in H^1_\loc(\R^d)$ the radiating solution to~\eqref{eq:T}, and \marker{$R$ so large that $\overline{D} \subset B_R$}. 
In particular, $v|_{D} \in H^1(B_R)$ depends continuously on $f\in L^2(D)$.  
Compactness of the embedding of $H^1(B_R)$ in $L^2(B_R)$ hence shows compactness of the sesquilinear form on the right of~\eqref{eq:aux252} on \marker{$L^2(B_R) \times L^2(B_R)$}. 
\marker{This motivates to define the self-adjoint positive definite operator $T_0: \, f \mapsto k^2 (q_1-q_2) f$ and the compact self-adjoint operator $K: \, f \mapsto k^2 (K_0+K_0^\ast)/2$ with $K_0 f = (q_1-q_2) v$ for $v \in H^1_\loc(\R^d)$ solving~\eqref{eq:T}.} 

(b) We merely show that $q_1 \geq q_2$ in $\R^d$ implies that $\Im \lambda_j >0$ and $\Re \lambda_j >0$ for $j$ large enough.
(The case $q_1 \leq q_2$ is handled analogously.)
Note that we already know from Lemma~\ref{th:diffSign} that $\Im \lambda_j \geq 0$.
If $\Im \lambda_j$ vanishes, then part (4) of the proof of Lemma~\ref{th:diffSign} shows that the far field $v_j^\infty$ of the solution $v_j$ to~\eqref{eq:T} with right hand side $-k^2 (q_1-q_2) T H_2 \psi_j$ vanishes.
In particular, the factorization and the eigenvalue equation imply that
\[
  \mathcal{S}_2^\ast (F_1 - F_2)\psi_j = H_2^\ast T_{1\& 2} H_2 \psi_j = \lambda_j \psi_j = w_j^\infty = 0,
\]
such that $\lambda_j$ vanishes.
Thus, no eigenvalue can belong to $\R \setminus \{ 0 \}$.
Assume next for contradiction that $\Re \lambda_j = 0$ for infinitely many $j\in\N$.
Without loss of generality, we can hence assume that $\Re \lambda_j = 0$ for all $j>N\in\N$.
As $H_2$ is injective by Lemma~\ref{th:diffSign}, the closure of $\mathrm{span}\{ H_2 \psi_j , \, j \in \N\}$ in $L^2(D)$ has infinite dimension.
Thus,~\eqref{eq:realBound} implies for the infinite-dimensional set of unit vectors $\varphi_j = H_2 \psi_j / \| H_2 \psi_j \|_{L^2(D)}$ that $0<c_0 \leq (-K \varphi_j,\varphi_j)_{L^2(D)}$.
The compactness argument from the end of part (a) again yields a contradiction.
%
\end{proof}

The last result shows the following monotonicity result:
The assumption that $q_1-q_2 \gtreqless 0$ implies, roughly speaking, that the real part of all but a finite number of the eigenvalues of $\mathcal{S}_2^\ast (F_1-F_2)$ is positive (or negative) as well.
If $\supp(q_1) = \supp(q_2)$ we will substantially refine this result in the next section by proving an even stronger monotonicity between the values of $q_1 - q_2$ on the boundary of $\supp(q_{1,2})$ and the real parts of the eigenvalues of $\mathcal{S}_2^\ast (F_1-F_2)$ (see Theorem~\ref{lemma:E1}).

Moreover, if $1+q_2$ is the refractive index of a known background medium that is perturbed by $q_1$, the results from this section show the following characterization of $\supp(q_1-q_2)$ via $F_1$ or via $\mathcal{S}_2^\ast(F_1-F_2)$, as $F_2$ and $\mathcal{S}_2$ can be computed from $q_2$ (see also~\cite{Cakon2015} for related results).
To this end, we denote by $G(\cdot,z) \in H^1_\loc(\R^d \setminus \{z \})$ the Green's function for the known background medium $1+q_2$, i.e., the distributional solution to
\begin{equation}\label{eq:GrensFuncQ1}
  \Delta G(\cdot, z) + k^2 (1+q_2) G(\cdot, z) = -\delta_z \in \R^d
\end{equation}
that satisfies Sommerfeld's radiation condition~\eqref{eq:SRC}.
(In~\eqref{eq:GrensFuncQ1}, $\delta_z$ is the Dirac distribution at $z\in \R^d$.)
This radiation condition is well-defined since $(\Delta + k^2) G(\cdot,z) = 0$ outside of $\supp(q_2)\cap \{z\}$, such that $G(\cdot,z)$ is a smooth solution to the Helmholtz equation outside some ball $B(0,R)$ with $R>0$ large enough.
In consequence, $G(\cdot,z)$ possesses a far field $G^\infty(\cdot, z)$.

\begin{theorem}\label{th:diffF}
Assume that $q_{1,2}\in L^\infty_\cmp(\R^d,\R)$ are two different real-valued contrasts such that either $q_1\geq q_2$ in $\R^d$ and $q_1- q_2 \geq c_0>0$ in $\supp(q_1-q_2)$ or else $q_1 \leq q_2$ in $\R^d$ and $q_2 - q_1\leq c_0>0$ in $\supp(q_1-q_2)$.
Further, set $M = \mathcal{S}_2^\ast(F_1-F_2)$.
Then $z\in \R^d$ belongs to $\supp(q_1-q_2)$ if and only if $\mathcal{S}_2^\ast G^\infty(\cdot,z)$ belongs to the range of the square root of the self-adjoint, compact, and non-negative operator $M_\sharp = |\Re M| + \Im M$ on $L^2(\S^{d-1})$.
\end{theorem}
\begin{proof}
We merely treat the case that $q_1\geq q_2$ in $\R^d$ and $q_1- q_2 \geq c_0>0$ in $\supp(q_1-q_2)$; the other case follows analogously.
Lemma~\ref{th:diffSign} and~\ref{th:eigsSign} show that $H_2$ is compact and injective and that $T_{1\& 2}$ is injective with non-negative imaginary part; moreover, $\Re T_{1\& 2}$ is a compact perturbation of a coercive operator, as shown in the proof of Lemma~\ref{th:eigsSign}.
The factorization $\mathcal{S}_2^\ast(F_1-F_2) = H_2^\ast T_{1\& 2} H_2$ then shows that the ranges of $H_2^\ast$ and of the square root of $M_\sharp = |\Re M| + \Im M$ are equal, see Theorem 2.15 in~\cite{Lechl2009a}.
(Since $M_\sharp$ is non-negative, compact and self-adjoint, such a square root can be defined, e.g., using a functional calculus for compact and self-adjoint operators.)
In addition, Theorem 4.5 in~\cite{Cakon2015} shows that $\mathcal{S}_2^\ast G(\cdot,z)$ belongs to the range of $H_2^\ast$ if and only if $z \in \supp(q_1-q_2)$, which yields the claim.
\end{proof}

\section{Factorization via Dirichlet-to-Neumann operators}\label{se:facDtN}
In this section we prove a second factorization of $\mathcal{S}_2^\ast (F_1-F_2)$ using Dirichlet-to-Neumann (DtN) operators.
This factorization requires more smoothness than the one from the last section; under these assumptions, however, it shows a monotonicity relation between the real part of all but a finite number of the eigenvalues of $\mathcal{S}_2^\ast (F_1-F_2)$ and the sign of the restriction of $q_1-q_2$ to the boundary of, roughly speaking, the union of the joint support of $q_{1,2}$.

Despite we require more smoothness later on, assume for the moment that the contrasts $q_{1,2} \in L^\infty_\cmp(\R^d)$ are bounded and measurable with supports $\overline{D_{1,2}}:=\supp q_{1,2} \subset \R^d$ for Lipschitz domains $D_{1,2}$.
Further, we set $G$ to be the unbounded connected component of the complement of $D_1 \cup D_2$, define $D_{1\&2} = \R^d \setminus G$ (this is the smallest set without holes containing $D_1$ and $D_2$), and assume that $D_{1\&2}$ is a Lipschitz domain as well, see Figure~\ref{fig:0}.

\begin{figure}[t!!!!!h!!!b!!]
  \centering
  \includegraphics[width=0.3\linewidth]{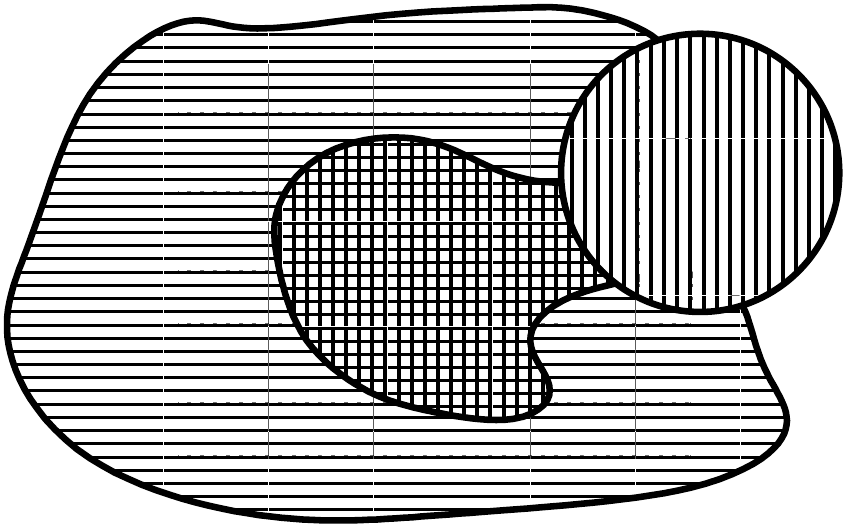}
  \begin{picture}(00,00)
    \put(-5.2,0.6){$D_1$}
    \put(0.0,1.7){$D_2$}
  \end{picture}
  \caption{Sketch of domains $D_1$ (left, horizontal lines) and $D_2$ (right, vertical lines); $D_{1\&2}$ is the union of $D_1$ and $D_2$ with the crossed region in the middle.}
  \label{fig:0}
\end{figure}
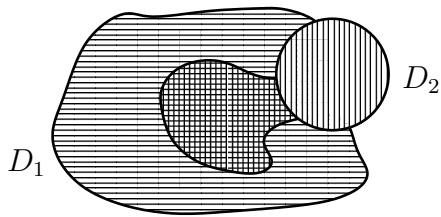

We assume that $k^2$ is not an interior Dirichlet eigenvalue of the negative Laplacian in $D_{1,2}$ or $D_{1\&2}$ and rely on various interior and exterior DtN operators for the Helmholtz equation.

For the homogeneous Helmholtz equation, and \marker{$D_j$} equal to either $D_{1,2}$ or $D_{1\&2}$,
\begin{equation}\label{NN1}
  N^{\out}_{D_j}: \, H^{1/2}(\partial D_j)\to H^{-1/2}(\partial \marker{D_j}), \qquad \psi \mapsto \left. \frac{\partial v}{\partial \nu} \right|_{\partial \marker{D_j}},
\end{equation}
maps Dirichlet boundary values to the Neumann boundary values of the unique radiating solution to the exterior boundary value problem $\Delta v + k^2 v = 0$ in $\R^d \setminus \overline{\marker{D_j}}$ subject to $v|_{\partial \marker{D_j}} = \psi$.
Note that $\nu$ is, as in the previous sections, the outer unit normal to \marker{$D_j$}.
Further, for \marker{$D_j$ equal to $D_{1,2}$ or $D_{1\&2}$ and $q_\ell$ equal to $q_{1,2}$ or $q_1+q_2$}, 
\begin{equation}\label{NN0}
N^\inn_{D_j,q_\ell}: \, H^{1/2}(\partial \marker{D_j})\to H^{-1/2}(\partial \marker{D_j}), \qquad \psi \mapsto \left. \frac{\partial v}{\partial \nu} \right|_{\partial D_j},
\end{equation}
maps Dirichlet boundary values to the Neumann boundary values of the unique radiating solution to the corresponding interior boundary value problem $\Delta v + k^2 (1+ \mathbbm{1}_{D_j} q_\ell) v = 0$ in $D_j$ subject to $v|_{\partial D_j} = \psi$.
(See~\cite[Ch.~4]{McLea2000} for such existence results.)
By $N^\inn_{D_j,0}$ we denote the corresponding operators for the Helmholtz equation $\Delta v + k^2 v = 0$ in $D_j$ without contrast function, i.e., for constant coefficients.
All these interior boundary value problems are assumed to be uniquely solvable.

Note that the difference $N^\inn_{D_j,q_j} - N^\out_{D_j}:\, H^{1/2}(\partial \marker{D_j})\to H^{-1/2}(\partial \marker{D_j})$ then maps Dirichlet trace values $\psi$ to the jump $\phi$ across $\partial D_j$ of the normal derivative of the unique radiating solution $u \in H^1_\loc(\R^d \setminus \partial D_j)$ to the transmission problem
\begin{multline} \label{eq:transmissionProblem}
  \Delta u + k^2 (1+\mathbbm{1}_{D_j} q_j) u = 0 \quad \text{in } \R^d \setminus \partial D_j, \\
  [u]_{\partial D_j} = 0 \quad \text{in } H^{1/2}(\partial D_j),
  \qquad
  \left[ \frac{\partial u}{\partial \nu}\right]_{\partial D_j} = \phi \in H^{-1/2}(\partial D_j).
\end{multline}
(See~\cite[Ch.~4]{McLea2000} for existence theory to this problem; $[v]_{\partial D_j}$ denotes the jump of $v$ from the outer to the inner trace on $D_j$.)
Indeed,
\begin{equation}\label{NN}
  N^\inn_{D_j,q_j} \psi - N^\out_{D_j} \psi = \left. \frac{\partial u}{\partial \nu}\right|_{\partial D_j}^- - \left. \frac{\partial u}{\partial \nu}\right|_{\partial D_j}^+
  = \left[ \frac{\partial u}{\partial \nu}\right]_{\partial D_j}
  = \phi \quad \text{in } H^{1/2}(\partial D_j).
\end{equation}
As the transmission problem~\eqref{eq:transmissionProblem} is uniquely solvable, the mapping $\phi \mapsto \psi$ is bounded from $H^{-1/2}(\partial D_j)$ into $H^{1/2}(\partial D_j)$ and defines the inverse to $\psi \mapsto N^\inn_{D_j,q_j} \psi - N^\out_{D_j}  \psi$.
Thus, $N^\inn_{D_j,q_j} - N^\out_{D_j} $ is boundedly invertible from $H^{1/2}(\partial \marker{D_j})$ into $H^{-1/2}(\partial \marker{D_j})$.

We now prove a relation between DtN operators and far-field operators $F_{1,2}$ where the link between far fields on the sphere and quantities on the boundary of the scatterer is played by the operator $L_j : \, L_2(\S^{d-1}) \to H^{1/2}(\partial D_j)$ defined by
\begin{equation}\label{lstar}
(L_j g)(y)= \int_{\S^{d-1}} e^{ik\, y\cdot \hat{x}} g(\hat{x}) \dS{(\hat{x})}, \qquad g\in L_2(\S^{d-1}), \, y \in \partial D_j.
\end{equation}
This is hence the restriction of a Herglotz wave function $v_g$ from~\eqref{eq:herglotzOperator} to $\partial D_j$ \marker{where $D_j \in \{ D_{1,2}, D_{1\& 2} \}$}.
Its $L^2$-adjoint is $L_j^\ast: \, H^{-1/2}(\partial D_j) \to L^2(\S^{d-1})$ mapping $v$ to $\hat{x} \mapsto\int_{\partial D_j} e^{-ik\, \hat{x} \cdot y} v(y) \, \dS{(y)}$.

\begin{theorem}\label{t21}
For $j=1,2$, the far-field operator $F_j$ satisfies
\begin{equation}\label{FFF}
 F_j = L_j^\ast \, (  N^\inn_{D_j,0} - N^\out_{D_j} ) (N^\inn_{D_j,q_j} - N^\out_{D_j} )^{-1} (N^\inn_{D_j,0} - N^\inn_{D_j,q_j}) \, L_j.
\end{equation}
\end{theorem}
\begin{proof}
We restrict ourselves to $j=1$, omit this index in this proof for all operators, fields, and domains, and denote by $\Phi$ the radiating fundamental solution of the Helmholtz equation with wave number $k^2$. By Green's representation theorem, the scattered wave $u^s$ for an incident Herglotz wave function $u^i(x) = \int_{\S^{d-1}} \exp(\mathrm{i} k\, x \cdot \theta) g(\theta) \dS{(\theta)}$ can be written as
$$
u^s(x)=\int_{\partial D}\left (\frac{\partial \Phi(x-y)}{\partial \nu(y)}u^s(y)-\Phi(x-y)\frac{\partial u^s}{\partial \nu} (y)\right ) \dS{(y)}, \quad x\in \mathbb R^d\setminus \overline{D}.
$$
Green's second identity applied to $\Phi(x,\cdot)$ and the solution of the Helmholtz equation in $D$ with the Dirichlet data $u^s|_{\partial D}$ at the boundary implies that
\begin{eqnarray}\nonumber
\int_{\partial D}\frac{\partial \Phi(x-y)}{\partial \nu(y)} u^s(y) \dS{(y)}
= \int_{\partial D} \Phi(x-y) N^\inn_{D,0} u^s(y) \dS{(y)}, \quad x\in \mathbb R^d\setminus \overline{D}.
\end{eqnarray}
Thus,
\begin{eqnarray}\nonumber
 u^s(x) = \int_{\partial D} \Phi(x-y) \ ( N^\inn_{D,0} u^s- N^\out_{D} u^s) (y) \dS{(y)}, \quad x\in \mathbb R^d\setminus \overline{D}.
\end{eqnarray}
As the far field of $\Phi(\cdot-y)$ equals $\hat{x} \mapsto \exp(-\i k \, \hat{x} \cdot y)$, the far field $u^\infty$ of $u^s$ satisfies
\begin{equation}\label{scatc2}
  u^\infty
  =  L^\ast ( N^\inn_{D,0} u^s - N^\out u^s)
  \quad
  \text{in } L^2(\S^{d-1}).
\end{equation}
It remains to express $u^s$ on $\partial D$ via the Herglotz wave operator $L g$ from~\eqref{lstar} that defines the restriction of the incident field $u^i$ to $\partial D$.
Note that the total field $u^i + u^s$ satisfies $N^\inn_{D,q} (u^i + u^s) = \partial u^i / \partial \nu + \partial u^s / \partial \nu$ in $H^{-1/2}(\partial D)$.  
Further, $\partial u^i / \partial \nu = N^\inn_{D,0} u^i$ whereas $\partial u^s / \partial \nu = N^\out_D u^s$, such that we conclude that 
\begin{align*}
  (N^\inn_{D,q} - N^\out_{D}) u^s |_{\partial D}
  = \left(N^\inn_{D,0} - N^\inn_{D,q} \right) u^i |_{\partial D} 
  & =  \left(N^\inn_{D,0} - N^\inn_{D,q}\right) L g
\end{align*}
holds in $H^{-1/2}(\partial D)$.
The bounded invertibility of $N^\inn_{D,q}-N^\out_{D}$ together with (\ref{scatc2}) now completes the proof.
\end{proof}

The last proof can be modified in the following way:
If $h$ denotes the restriction of an incident Herglotz wave function $u^i$ to $\partial D_{1\&2}$ (see Figure~\ref{fig:0}), and if $u^s_j$ denotes the solution to the scattering problem for contrast $q_j$, then $N^\inn_{D_{1\&2},q_j} h = \partial u_j^s / \partial \nu$ as well as $N^\inn_{D_{1\&2},0} h = \partial u^i / \partial \nu$ holds in $H^{-1/2}(\partial D_{1\&2})$.
The last proof hence also shows the following result.

\begin{corollary}\label{t22}
For $j=1,2$, the far-field operator $F_j$ satisfies
\begin{equation}\label{FFF3}
 F_j = L_{1\&2}^\ast \, ( N^\inn_{D_{1\&2},0} - N^\out_{D_{1\&2}} ) (N^\inn_{D_{1\&2},q_j} - N^\out_{D_{1\&2}} )^{-1} (N^\inn_{D_{1\&2},0} - N^\inn_{D_{1\&2},q_j}) \, L_{1\&2}.
\end{equation}
\end{corollary}

The following property of the outer operators $L_{1\&2}$ and $L_{1\&2}^\ast$ is well-known, see~\cite{Laksh2013, Kirsc2008}, and holds of course also for $D_{1,2}$ instead of $D_{1\&2}$.

\begin{lemma}\label{llstar}
If $-k^2$ is not an eigenvalue of the negative Dirichlet-Laplacian in $D_{1\&2}$, then both operators $L_{1\&2}: \, L^2(\S^{d-1}) \to H^{1/2}(\partial D_{1\&2})$ and $L_{1\&2}^\ast: \, H^{-1/2}(\partial D_{1\&2}) \to L^2(\S^{d-1})$ are injective and their ranges are dense.
\end{lemma}
%

The last lemma shows that $F_j$ can be written as $F_j = L_{1\&2}^\ast \, M_j \, L_{1\&2}$ with
\begin{equation}\label{eq:Mj}
  M_j = ( N^\inn_{D,0} - N^\out_{D} ) (N^\inn_{D,q_j} - N^\out_{D} )^{-1} (N^\inn_{D,0} - N^\inn_{D,q_j})
\end{equation}
for $j=1,2$ by~(\ref{FFF3}).
Thus, $\mathcal S_2^\ast(F_1-F_2)$ is representable in the form
\begin{align}
  \mathcal{S}_2^\ast(F_1-F_2)
  & = \left(I - 2\i k |\gamma_d|^2 \, F_2^\ast\right) (F_1-F_2) \nonumber \\
  & = \left(I - 2\i k |\gamma_d|^2 \, L_{1\&2}^\ast M_2^\ast L_{1\&2}\right) \left(L_{1\&2}^\ast [M_1-M_2] L_{1\&2} \right) \label{eq:factoCompl} \\
  & = L_{1\&2}^\ast \big(\underbrace{M_1 - M_2 - 2\i k |\gamma_d|^2 M_2^\ast L_{1\&2} L_{1\&2}^\ast [M_1-M_2]}_{=: M_{1\&2}} \big) L_{1\&2}, \nonumber
\end{align}
with a bounded operator $M_{1\&2}$ mapping $H^{1/2}(\partial D_{1\&2})$ into $H^{-1/2}(\partial D_{1\&2})$.
The latter middle operator can be analyzed by pseudo-differential calculus.
To this end, we suppose from now on that the two contrasts $q_{1,2}$ are infinitely often differentiable functions inside their joint support $\overline{D}:=\supp q_{1,2} \subset \R^d$, and that all partial derivatives possess continuous extensions to $\overline{D}$.
The domain $D$ is moreover assumed to be smooth and bounded with connected complement.
(These assumptions avoid technicalities and imply in particular that $D_{1\&2} = D$.
It would be sufficient to assume that $q_{1,2}$ are both $C^3(\overline{D})$ and that $D$ is a domain of class $C^4$, see~\cite{Laksh2013b}.)
Writing $L = L_{1,2}$, the factorization in~\eqref{eq:factoCompl} hence simplifies to
\begin{align} \label{eq:factoSimpl}
  \mathcal{S}_2^\ast(F_1-F_2)
  & = L^\ast M_{1\&2} L
  = L^\ast \big( M_1 - M_2 - 2\i k |\gamma_d|^2 M_2^\ast L_{1\&2} L^\ast [M_1-M_2] \big) L.
\end{align}

Let $(y_1,...y_{d-1})^\top$ be local coordinates on $\partial D$ with dual variables $(\xi^\ast_1, ...,\xi^\ast_{d-1})$ and let $\sum_{i,j=1}^{d-1} g_{i,j}(y) \d{y_i} \d{y_j}$ be the first fundamental form on $\partial D$. Then $|\xi^\ast|= \big(\sum_{i,j=1}^{d-1} g^{i,j}(y) \xi^\ast_i\xi^\ast_j \big)^{1/2}$ is the length of the covector in the cotangent bundle $T^\ast(\partial D)$.

\begin{lemma}\label{mprl}
Suppose that $k^2$ is such that the DtN operators $N^\inn_{D,q_j}$, $j=1,2$, and $N^\inn_{D,0}$ are well-defined.

(a) Both operators $N^\inn_{D,q_j}$ and $N^\inn_{D,0}$ are elliptic pseudo-differential operators of order one and self-adjoint from $H^{1/2}(\partial D)$ into $H^{-1/2}(\partial D)$. The principal symbols of both operators equal $|\xi^\ast|$.

(b) The operator $N^\out_{D}$ is an elliptic pseudo-differential operator of order one with principal symbol $-|\xi^\ast|$. For every $\psi\neq 0$ in $H^{1/2}(\partial D)$,
\begin{equation}\label{kvf}
\Im(N^\out \psi,\psi)_{L^2(\partial D)} = k \, |\gamma_d|^2 \, \int_{\S^{d-1}} |v^\infty|^2 \dS{} >0,
\end{equation}
where $v^\infty$ is the far-field amplitude of the solution $v$ of the exterior Dirichlet scattering problem in $\R^d \setminus \overline{D}$ with Dirichlet boundary data $\psi \in H^{1/2}(\partial D)$.

(c) If $q_j$ does not vanish on the boundary $\partial D$, then the operator $N^\inn_{D,0} -N^\inn_{D,q_j}$ from in (\ref{NN}) is an elliptic pseudo-differential operator of order minus one with  principal symbol $(x,\xi^\ast) \mapsto k^2 q_j(x)/ (2|\xi^\ast|)$ for $(x,\xi^\ast) \in \partial D \times T^\ast(\partial D)$.

(d) If $q_j$ is identically zero on the boundary $\partial D$ and its normal derivative does not vanish anywhere on the boundary, then the operator $N^\inn_{D,0} -N^\inn_{D,q_j}$ from in (\ref{NN}) is an elliptic pseudo-differential operator of order minus two with principal symbol $(x,\xi^\ast) \mapsto -k^2 (\partial q_j(x)/\partial \nu)/ (4|\xi^\ast|)$ for $(x,\xi^\ast) \in \partial D \times T^\ast(\partial D)$.
More generally, if we suppose that exists $m\in \N_0$ such that
\begin{equation}\label{1301A}
\frac{\partial^i q_j(x)}{\partial \nu^i} \equiv  0, ~  i=0, \ldots, m-1, ~ \frac{\partial^m q_j(x)}{\partial \nu^m} \neq 0  \quad \text{for } x \in \partial D,
\end{equation}
then there is a constant $\rm{const_m}>0$ such that $N^\inn_{D,0} -N^\inn_{D,q_j}$ has principal symbol $(x,\xi^\ast) \mapsto (-1)^m k^2 \, \mathrm{const}_m \, (\partial^m q_j(x) / \partial \nu^m)/ |\xi^\ast|^{m+1}$ for $(x,\xi^\ast) \in \partial D \times T^\ast(\partial D)$.
\end{lemma}
 \begin{proof} The first statement and the expression for the symbols of $N^\inn_{D,q_j}$, $N^\inn_{D,0}$, and $N^\out_{D}$ are well known, see more details in \cite{Laksh2013b}.
The formula on the left of (\ref{kvf}) is a consequence of Green's first identity and the definition of the far field, compare~\eqref{eq:imagT}; positivity of the left-hand side is a consequence of Rellich's lemma. Two last statements can be found in \cite[lemma 1.1]{Laksh2013b}. It is justified by calculating the first three terms of the full symbols of $N^\inn_{D,0}$ and $N^\inn_{D,q_j}$ (the differences of the first two terms of the symbols vanishes).
The proof of item (d) consists in computing the full symbol of the pseudo-differential operators $N^\inn_{D,0}$ and $N^\inn_{D,q_j}$.
This procedure is described in detail in Sections 3 and 4 of \cite{Laksh2013b} and has been justified in \cite{vainberg1967uniformly}, see also~\cite[Ch.VII]{eskin2011lectures} and~\cite{lee1989determining}.
Note that the coefficient $\rm{const_m}$ of the principal symbol is calculated rigorously in \cite{Laksh2013b} for $m=0$ and $m=1$ only.
For general $m>0$, calculating $\rm{const_m}$ reduces to calculating two determinants of a band matrix of size $m\times m$ and band width two; we omit this calculation since 
it requires a significant amount of notation that is not going to be used again.
\end{proof}

The factorization of $M_j = ( N^\inn_{D,0} - N^\out_{D} ) (N^\inn_{D,q_j} - N^\out_{D} )^{-1} (N^\inn_{D,0} - N^\inn_{D,q_j})$ from Lemma~\ref{t22} into pseudo-differential operators with principal symbols introduced in the last lemma allows to compute the principal symbol of $M_{1\&2} = M_1 - M_2 - 2\i k |\gamma_d|^2 M_2^\ast L  L^\ast [M_1-M_2]$ from~\eqref{eq:factoCompl}.
Note that $L \, L_{1\&2}^\ast$ is compact from $H^s(\partial D)$ into $H^t(\partial D)$ for arbitrary $s,t \in\R$, such that $M_2^\ast L L^\ast [M_1-M_2]$ is bounded from $H^{1/2}(\partial D)$ into $H^t(\partial D)$ for all $t \in \R$.
In particular, this operator is irrelevant for computing the principal symbol of $M_{1\&2}$.
As the principal symbols of $N^\inn_{D,q_j}$ and $N^\inn_{D,0}$ equal $(x,\xi^\ast) \mapsto |\xi^\ast|$, as that of $N^\out_{D}$ equals $(x,\xi^\ast) \mapsto-|\xi^\ast|$, and as that of $N^\inn_{D,0} -N^\inn_{D,q_j}$ equals $(x,\xi^\ast) \mapsto k^2 q_j(x)/ (2|\xi^\ast|)$, the principal symbol of $M_{1\&2}$ equals
\begin{equation}
  \label{eq:priSym}
  (x,\xi^\ast) \mapsto
  \frac{2 |\xi^\ast|}{2|\xi^\ast|} \, \, k^2 \frac{q_1(x)-q_2(x)}{2|\xi^\ast|}
  = k^2 \ \frac{q_1(x)-q_2(x)}{2|\xi^\ast|}
  \quad \text{for } (x,\xi^\ast) \in \partial D \times T^\ast(\partial D).
\end{equation}

\begin{theorem}\label{lemma:E1}
(a)
If $q_1-q_2 < 0$ on $\partial D$, then $\mathcal S_2^\ast(F_1-F_2)$ has at most a finite number of eigenvalues $\lambda_j$ with positive real part.

(b)
If $q_1 - q_2 >0$ on $\partial D$, then $\mathcal S_2^\ast(F_1-F_2)$ has at most a finite number of eigenvalues $\lambda_j$ with negative real part.

(c)
If $q_1- q_2$ takes both positive and negative values on $\partial D$, then $\mathcal S_2^\ast(F_1-F_2)$ has infinitely many eigenvalues with both  positive and negative part.

(d)
In case that $q_1\equiv q_2$ at the boundary but (\ref{1301A}) holds for some $m>0$ then corresponding result (a), (b) or (c) holds depending on sign of the $m$th normal derivative.
\end{theorem}

\begin{remark}
Theorem~\ref{lemma:E1} holds irrespective of whether $k^2$ is such that the interior boundary value problems defining the DtN operators $N^\inn_{D,q_{1,2}}$ and $N^\inn_{D,0}$ from~\eqref{NN0} are uniquely solvable.
Indeed, by the continuous dependence of $F_{1,2}$ on $k$, such interior eigenvalues might flip the sign of the real part of at most finitely many eigenvalues, which does not influence finiteness or infiniteness of the corresponding sets of eigenvalues.
\end{remark}

\begin{proof}
(1) Let $q_1(x)-q_2(x)<0$ on $\partial D$.  Let $T^+=\overline{\mathrm{span}\{\varphi_j^+\}}$, where $\varphi_j^+$ are the orthonormal eigenfunctions of $\mathcal S_2^\ast(F_1-F_2)$ associated to eigenvalues $\lambda_j$ with positive real part $\Re \lambda_j \geq 0$.
To prove the first statement of the theorem, we need to show that the space $T^+$ is finite-dimensional.
To this end, we abbreviate the scalar product of $L^2(\S^{d-1})$ by $(\cdot,\cdot)$.

(2) By construction, we have that $\Re( \mathcal S_2^\ast(F_1-F_2)\varphi_j^+,\varphi_j^+) = \Re \lambda_j \geq 0$.
Orthogonality of the eigenfunctions $\varphi_j^+$ hence implies that
\begin{equation}\label{kvf1}
 \Re( \mathcal S_2^\ast(F_1-F_2)\varphi,\varphi)\geq0 \quad \text{for all } \varphi\in T^+.
\end{equation}
We next use the representation $\mathcal S_2^\ast(F_1-F_2)=L^\ast M_{1\&2} L$, where $M_{1\&2}$ is a pseudo-differential operator with the principal symbol $k^2 (q_1(x)-q_2(x)) / (2|\xi^\ast|)$ due to~\eqref{eq:priSym}.
For all $\varphi\in L^2(\S^{d-1})$, we have
\begin{equation*}
  (\mathcal S_2^\ast(F_1-F_2)\varphi,\varphi)
  = (M_{1\&2}\psi,\psi)_{L^2(D)} \quad \text{for } \psi=L\varphi \in H^{1/2}(\partial D).
\end{equation*}
Since $M_{1\&2}$ is an elliptic operator of order one with a negative principal symbol, there is $c_0>0$ such that
\begin{equation}\label{kvf2}
\Re(M_{1\&2}\psi,\psi)\leq -c_0\|\psi\|_{H^{1/2}(\partial D)}^2 + C \|\psi\|_{L^2(\partial D)}^2,
\end{equation}
and therefore
\begin{equation}\label{kvfMMM}
0 \leq \Re(\mathcal S_2^\ast(F_1-F_2)\varphi,\varphi)\leq -c_0\| L\varphi \|_{H^{1/2}(\partial D)}^2 + C \| L\varphi \|_{L^2(\partial D)}^2  \quad \text{for all } \varphi\in T^+.
\end{equation}
Thus, for all $\psi$ in the closure of $L (T^+) = \big\{ \psi = L \varphi \text{ for some } \varphi \in T^+ \big\}$ in the norm of $H^{1/2}(\partial D)$ there holds the inequality
\begin{equation}\label{eq:aux510}
  \| \psi \|_{H^{1/2}(\partial D)}^2
  \leq  \frac{C}{c_0} \| \psi \|_{L^2(\partial D)}^2, 
  \qquad 
  \psi \in \overline{L (T^+)}.
\end{equation}
On any infinite-dimensional subset of $H^{1/2}(\partial D)$, the $H^{1/2}(\partial D)$-norm cannot be estimated from above by the $L^2(\partial D)$-norm due to the open mapping theorem.
Consequently,~\eqref{eq:aux510} implies that the linear space $\overline{L(T^+)}$ is finite-dimensional.
Now, Lemma \ref{llstar} implies that the space $T^+$ is finite-dimensional, too, such that the first statement of the theorem is proved.

(3) To prove the second statement, one needs to replace $T^+$ by $T^-=\overline{{\rm span}\{\varphi_j^-\}}$, where $\varphi_j^-$ are the eigenfunctions corresponding to eigenvalues $\lambda_j$ with negative real part, and use the positivity of the principal symbol of $M_{1\&2}$.
Let us hence prove the last statement by combining the above technique with a localization argument.

(4) Assume hence that $q_1-q_2$ takes both positive and negative values on $\partial D$ and that the space $T^- = {\rm span}\{\varphi_j^-\}$,  defined as above, is finite-dimensional.
Similarly to (\ref{kvf1}), we have that $\Re( \mathcal S_2^\ast(F_1-F_2)\varphi,\varphi)\geq0$ for all $\varphi$ in the orthogonal complement $(T^-)^\bot$ of $T^-$, and therefore
\begin{equation} \label{kvf5}
  ( \mathcal S_2^\ast(F_1-F_2)\varphi,\varphi) = \Re(M_{1\&2} L \varphi, L \varphi)_{L^2(\partial D)} \geq 0 \quad \text{for all } \varphi \in (T^-)^\bot.
\end{equation}
The smoothness of $q_{1,2}$ implies that there is an $\varepsilon>0$ so small that the set $\Gamma^-= \{ x \in \partial\Omega, \, q_1(x) - q_2(x)<\varepsilon\}$ is not empty.
Let $\chi$ be an infinitely smooth function included in $C^\infty(\overline{D})$ such that $0 \leq \chi \leq 1$ and such that $\chi \equiv 1$ in a $d$-dimensional neighborhood $U$ of $\Gamma^-$ in $D$ with $U \bigcap \{ x \in \partial\Omega, \, q_1(x) - q_2(x) \geq 0 \} = \emptyset$.
It is always possible to choose $\chi$ such that both DtN operators $N^\inn_{D,\chi q_j}$, $j=1,2$, are well-defined between $H^{\pm 1/2}(\partial D)$.

For $\psi \in H^{1/2}(\partial D)$, consider now solutions $v,w \in H^1(D)$ of the boundary value problem
\[
  \Delta v + k^2 (1+q_j) v = 0 \text{ in } D,
  \quad
  \Delta w + k^2 (1+\chi q_j) w = 0 \text{ in } D,
  \quad
  v = w = \psi \text{ on } \partial D,
\]
such that $N^\inn_{D,q_j} \psi = \partial v /\partial \nu$ and $N^\inn _{D,\chi q_j}\psi = \partial w / \partial \nu$ holds in $H^{-1/2}(\partial D)$.
The difference $\phi = N^\inn_{D,q_j} \psi - N^\inn _{D,q_j\chi} \psi$ hence equals the Neumann boundary values of $z=v-w \in H^1(D)$,
\[
  \Delta z + k^2 (1+q_j) v = k^2 (\chi-1) q_j w \text{ in } D,
  \quad
  z = 0 \text{ on } \partial D.
\]
As $\chi-1$ vanishes in the neighborhood $U$ of $\Gamma^-$, standard boundary estimates for the solutions of elliptic equations show that $\| z \|_{H^\ell(U)} \leq C(\ell) \| \psi \|_{H^{1/2}(\partial D)}$ for all arbitrary $\ell\in\N$, as long as $\psi$ is supported in $\Gamma^-$.
Thus, we introduce $\tilde H^{1/2}(\Gamma^-) = \big\{ \psi \in H^{1/2}(\Gamma^-), \, \supp(\psi) \subset \overline{\Gamma^-} \big\}$ and conclude that $\psi \mapsto \big(N^\inn_{D,q_j} - N^\inn _{D,q_j\chi}\big) \psi$ is bounded from $\tilde H^{1/2}(\Gamma^-)$ into $H^t(\Gamma^-)$ for arbitrary $t$.
(We implicitly extend functions in $\tilde H^{1/2}(\Gamma^-)$ by zero to elements of $H^{1/2}(\Gamma)$.)
%
%
If we merely consider $\psi \in \tilde H^{1/2}(\Gamma^-)$, then estimate (\ref{kvf2}) consequently not only holds for $M_{1\&2}$ but also for $M_{1\&2}'$, defined by replacing $q_1$ and $q_2$ in $M_{1,2}$ by $\chi q_1$ and $\chi q_2$, respectively.
As in part (2) of the proof, we conclude by (\ref{kvf5}) that
\[
  \|\psi\|_{H^{1/2}(\partial \Omega)}^2 \leq \frac{C}{c_0} \|\psi\|_{L_2(\partial\Omega)}^2  \quad \text{for } \psi\in \overline{L \big((T^-)^\bot\big)} \cap \tilde H^{1/2}(\Gamma^-),
\]
where the closure of $L((T^-)^\bot)$ is taken in the norm of $H^{1/2}(\Gamma)$. 
The latter inequality implies by the same arguments as in the end of part (2) that $\overline{L((T^-)^\bot)} \cap \tilde H^{1/2}(\Gamma^-)$ is finite-dimensional, such that $(T^-)^\bot$ must be finite-dimensional. 
This contradicts our initial assumption that  $T^-$ itself is a finite-dimensional subspace. 
The proof that $T^+$ can not be finite-dimensional follows analogously.
\end{proof}

\section{Applications}\label{se:appl}
As a corollary of the factorization of $F_1$ in Theorem~\ref{t21} we establish a factorization method for sign-changing contrasts.
As always in this section, we require that the DtN operators $N^\inn_{D,0}$ and $N^\inn_{D_j,q}$ from~\eqref{NN0} are well-defined for the considered contrast function $q$.

\begin{theorem} \label{th:facto}
Assume that $q$ is a real-valued contrast function supported in the smooth domain $\overline{D} \subset \R^d$ such that $q|_{\overline{D}}$ is a smooth function on $\overline{D}$.
Assume further that $\left. q \right|_{\partial D}$ is either strictly positive or strictly negative, and denote the far field operator associated to $q$ by $F = F_q$.
Additionally, suppose that $k^2$ is not a transmission eigenvalue of $D$, i.e., that there is no non-trivial pair $(v,w) \in H^1(D)^2$ such that $v-w \in H^2_0(D)$ solving
\begin{equation} \label{eq:puffpaff}
  \Delta v + k^2 (1+q) v = 0
  \quad \text{and} \quad
  \Delta w + k^2 w = 0 \quad \text{ in } D.
\end{equation}
Then $z \in \R^d$ belongs to $D$ if and only if $\phi_z(\hat{x}) := \exp(-\i k \, \hat{x}\cdot z) \in L^2(\S^{d-1})$ belongs to $\mathrm{Rg} \big( (F^\ast F)^{1/4} \big)$.
\end{theorem}
\begin{proof}
Theorem~\ref{t21} shows that $F= L^\ast M_1 L$, where $M_1: \, H^{1/2}(\partial D) \to H^{-1/2}(\partial D)$ can be represented as sum of a coercive operator plus a compact perturbation, since its principal symbol is either positive or negative due to Lemma~\ref{mprl}(a)-(c).
Recall that $M_1 = (N^\inn_{D,0} - N^\out_D) (N^\inn_{D,q} - N^\out_D)^{-1} (N^\inn_{D,0} - N^\inn_{D,q})$.
Our assumption that $k^2$ is not a transmission eigenvalue implies that $N^\inn_{D_j,0} - N^\inn_{D_j,q_j}$ is injective, since otherwise the difference of the corresponding interior Dirichlet boundary values belong to $H^2_0(D)$ and solve the two Helmholtz equations in~\eqref{eq:puffpaff}.
It is easy to see that $N^\inn_{D,0} - N^\out_D$ is injective, too, and we have already shown in the last section that $(N^\inn_{D,q} - N^\out_D)^{-1}$ is an isomorphism.
Thus, $M_1$ is injective as composition of three injective operators.
Lemma~\ref{th:diffSign} applied to $q_2 \equiv 0$ moreover shows that $\Im M_1$ is non-negative.
Further, Lemma~\ref{llstar} shows that $L: \, L^2(\S^{d-1}) \to H^{1/2}(\partial D)$ is injective with dense range.
As $\mathcal S = I + 2\i k |\gamma_d|^2 \, F$ is unitary, all hypotheses of Theorem 1.23 in~\cite{Kirsc2008} are satisfied such that this result implies that the ranges of $L^\ast$ and $(F^\ast F)^{1/4}$ are equal.
As $k^2$ is not an interior Dirichlet eigenvalue (since $N_{D,0}$ is assumed to be well-defined), Theorems 1.12 and 1.24 in~\cite{Kirsc2008} shows that the function $\phi_z$ belongs to the range of $L^\ast$ if and only $z \in D$, which shows the claim.
\end{proof}

\marker{The last theorem typically is exploited to define an indicator function for the support of the contrast function $q$ by noting that Picard's criterion~\cite{Kirsc2008} implies for the complete eigensystem $(\lambda_j,\varphi_j)_{j\in\N}$ of $F$ that 
\begin{equation} \label{eq:picard}
  z \mapsto \left[\sum_{j\in\Z} \frac{\left| \langle \phi_z, \, \varphi_j \rangle_{L^2(\S^{d-1})} \right|^2}{|\lambda_j|} \right]^{-1} > 0
  \qquad  
  \text{ if and only if }
  \qquad 
  z \in D,
\end{equation}
see~\cite{Kirsc2008}. 
Let us briefly illustrate the latter criterion numerically for the sign-changing contrast function $q_1$ shown in Figure~\ref{fig:X}(a) for far field data gained at wave number $k=5$ via 64 incident plane waves with uniformly distributed directions on the unit circle. 
As Figure~\ref{fig:X}(b) shows, the indicator function~\eqref{eq:picard} clearly indicates the shape of the contrast $q_1$. 
(We used Tikhonov regularization by with constant regularization parameter $10^{-8}$ for a numerical noise level above $10^{-6}$.) 
For comparison, we show in Figure~\ref{fig:X}(c) the behaviour of the same indicator function for a contrast $q_2$ with same support as $q_1$ but constant contrast equal to $0.7$.  
This comparison shows in particular that the indicator function for $q_2$ is almost flat in the interior, which, arguably, provides a better reconstruction. 
In both cases, however, the inverses of the plotted indicator functions are very small outside the support of the scatterers, which notably is the only property guaranteed by Theorem~\ref{th:facto} or~\eqref{eq:picard}.  
\begin{figure}[t!!!!!h!!!b!!]
  \centering
\begin{tabular}{ccc}
\vspace*{-3mm}\hspace*{-5.5mm}
\includegraphics[width=0.35\linewidth]{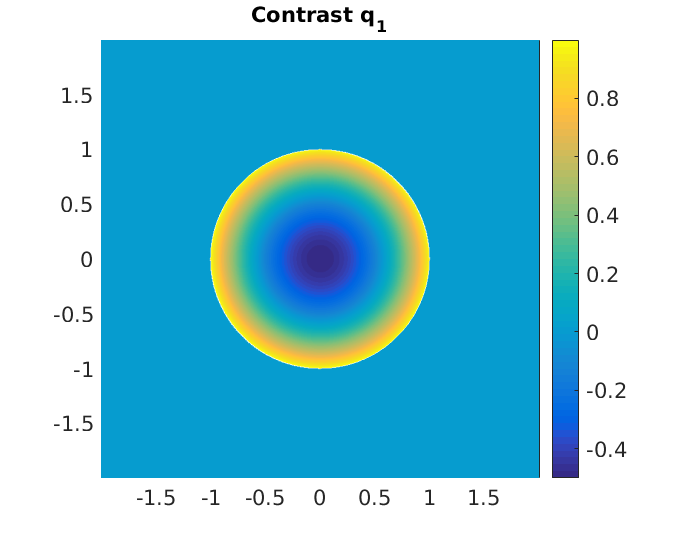}\hspace*{-6mm}
&\includegraphics[width=0.36\linewidth]{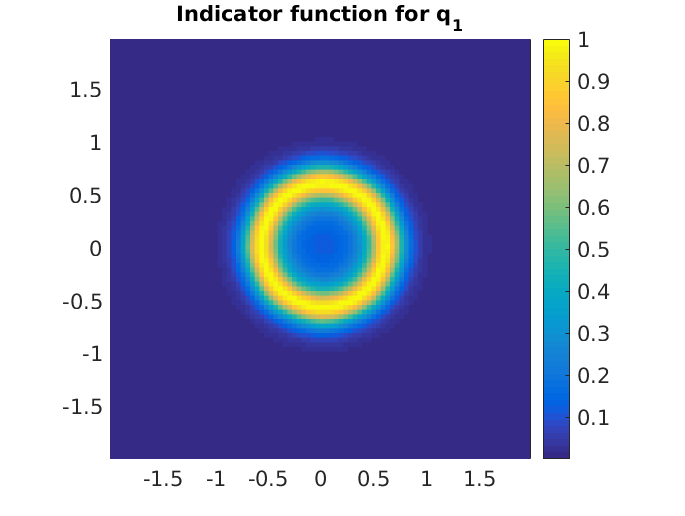}\hspace*{-6mm}
&\includegraphics[width=0.36\linewidth]{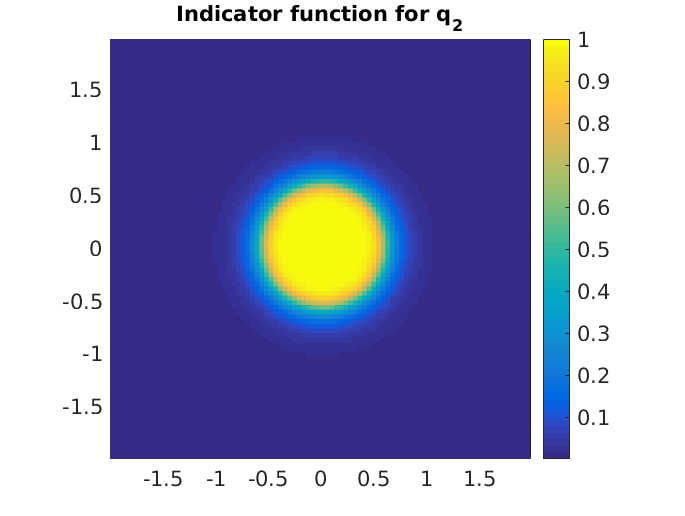}\\[1mm]
(a)&(b)&(c)\\[3mm]
\end{tabular}
\caption{\marker{(a) Contrast  $q_1$.
(b) Indicator function for $\supp(q_1)$ from the left of~\eqref{eq:picard}, scaled to maximal value one.
(c) Indicator function for $\supp(q_2)$, scaled to maximal value one. (Recall that $\supp(q_2)=\supp(q_1)$ and that $q_2|_{\supp(q_2)} = 0.7$.)}}
\label{fig:X}
\end{figure}
} 

As a further application, Theorem~\ref{lemma:E1} directly shows that the boundary values of a smooth contrast $q$ are uniquely defined by the far field operator $F_q$.

\begin{corollary}\label{th:uniqueBoundary}
If $D\subset \R^d$ is a known smooth domain and if $q: \, D \to \R$ is a smooth contrast function then $F=F_q$ uniquely determines  the boundary values $q|_{\partial D}$.
\end{corollary}
\begin{proof}
If $F_1 = F_2$ for two far field operators corresponding to two smooth contrast functions $q_{1,2}$, then $\mathcal S_2^\ast (F_1-F_2) = 0$, such that Theorem~\ref{lemma:E1} implies that $(q_1-q_2)|_{\partial D}$ cannot take positive or negative values.
\end{proof}

The following result considers a contrast $q$ with support $D$ that is analytic and possibly contains obstacles with prescribed \marker{non-absorbing} boundary conditions.

\begin{theorem}\label{1401A}
Suppose that the contrast function $q$ is analytic in its support \marker{$\overline{D}$ that contains finitely many connected obstacles $\overline\Omega \subset D$ of class $C^{0,1}$ with connected complement $D \setminus \overline\Omega$. 
Suppose moreover that the jump of $q$ across $\partial D$ is sign-definite and that the radiating scattered fields $u^s=u^s(\cdot,\theta) \in H^1_{\loc}(\R^d)$ for incident plane waves with direction $\theta \in \S^{d-1}$ solve $\Delta u^s + k^2 (1+q) u^s =- k^2 q u^i(\cdot,\theta)$ in $\R^d$, subject to transmission conditions $\left[ u^s \right]_{\partial D} = 0$, $\left[ {\partial u^s}/{\partial \nu} \right]_{\partial D} = 0 $, and either Dirichlet or Robin boundary conditions on $\partial \Omega$,  
\[
  u^s  = -u^i(\cdot,\theta) \text{ on } \partial\Omega
  \quad \text{or} \quad 
  \frac{\partial u^s}{\partial \nu} + \sigma u^s  = - \left[ \frac{\partial u^i(\cdot,\theta)}{\partial \nu} + \sigma u^i(\cdot,\theta) \right] \text{ on } \partial\Omega 
\]
for some real-valued function $\sigma \in L^\infty(\partial \Omega,\R)$. 
Additionally, suppose either that $k^2$ is not an interior Dirichlet or Robin eigenvalue of $\Omega$ for the negative Laplacian.
Then $q$ and the shape of all obstacles $\Omega$ included in $D$ are determined uniquely by the far field operator defined by the latter scattering problem.}
\end{theorem}
\begin{proof}
\marker{It is well-known that both the mixed scattering problem and the inhomogeneous medium scattering problem are uniquely solvable in $H^1_{\loc}(\R^d)$, and the corresponding proofs by variational methods extend to the scattering problem, see, e.g.,~\cite{Colto2013,Kirsc2013}.}  
As $D \in C^\infty$ is a smooth domain and $q|_{D}$ is restriction of an analytic function, the assumption on the jump of $q$ across $\partial D$ implies by Theorem \ref{lemma:E1} uniqueness of germs of $q$ in each boundary point on $\partial D$.
As, moreover, each germ of $q$ can be continued analytically into the whole of $\overline{D}$, the problem of identifying the shape of the obstacle is reduced to the problem of identifying the shape of obstacles in the known medium (produced by the mentioned germ of $q$), which has been solved for Dirichlet and Robin boundary conditions in~\cite{nachman2007imaging}
\end{proof}

Neglecting smoothness assumptions, the monotonicity between $(q_1-q_2)|_{\partial D}$ and the real parts of the eigenvalues of $(q_1-q_2)|_{\partial D}$ motivates the following algorithm to compute boundary values of a smooth contrast function $q$ when the smooth support $\overline{D} \subset \R^d$ of $q$ is a-priori known:
Computing far field operators for constant refractive index, determine in a first step constant upper and lower bounds for $q|_{\partial D}$.
Second, refine these bounds by decreasing/increasing the constant bounds locally on $\partial D$.
Let us for simplicity first investigate an algorithm determining constant bounds, before refining those in a second step.

\begin{lstlisting}[caption={Algorithm to find upper/lower bounds for the boundary values $q|_{\partial D}$ of real-valued contrast $q$ with $\supp(q) = \overline{D}$ from far field data $F_q$ with starting values $c_\ast<c^\ast \in \R$ and update parameter $t>0$.}, label=algo, frame=lines, mathescape=true]

$A = \mathcal{S}_{c_\ast\mathbbm{1}_D}^\ast (F_q - F_{c_\ast\mathbbm{1}_D})$;
if eigenvalues of $A$ tend to zero from the right // $\Rightarrow c_\ast < q|_{\partial D}$
    while eigenvalues of $A$ tend to zero from the right
        $c_\ast = c_\ast + t$;  // increase $c_\ast$
        $A = \mathcal{S}_{c_\ast\mathbbm{1}_D}^\ast (F_q - F_{c_\ast \mathbbm{1}_D})$;
    $c_\ast = c_\ast - t$;	
else
    while eigenvalues of $A$ $\text{do not}$ tend to zero from the right
        $b_\ast = b_\ast - t$;   // decrease $b_\ast$
        $A = \mathcal{S}_{b_\ast\mathbbm{1}_D}^\ast (F_q - F_{b_\ast \mathbbm{1}_D})$;

$A = \mathcal{S}_{c^\ast\mathbbm{1}_D}^\ast (F_q - F_{c^\ast\mathbbm{1}_D})$;
if eigenvalues of $A$ tend to zero from the left // $\Rightarrow c^\ast > q|_{\partial D}$
    while eigenvalues of $A$ tend to zero from the left
        $c^\ast = c^\ast - t$;  // decrease $c^\ast$
        $A = \mathcal{S}_{c^\ast\mathbbm{1}_D}^\ast (F_q - F_{c^\ast \mathbbm{1}_D})$;
    $c^\ast = c^\ast + t$;	
else
    while eigenvalues of $A$ $\text{do not}$ tend to zero from the left
        $c^\ast = c^\ast + t$;   // increase $c^\ast$
        $A = \mathcal{S}_{c^\ast\mathbbm{1}_D}^\ast (F_q - F_{c^\ast \mathbbm{1}_D})$;

return $c_\ast$, $c^\ast$;
\end{lstlisting}

\begin{corollary}\label{th:monoAlgo}
Under the assumptions of Corollary~\ref{th:uniqueBoundary}, the values $c_\ast$, $c^\ast$ returned by the algorithm in Listing~\ref{algo}  satisfy $c_\ast \leq  q|_{\partial D} \leq c^\ast$.
\end{corollary}

To show feasibility of the latter algorithm, we consider three contrasts $q_{\mathrm{c},\mathrm{v}, \mathrm{r}}$ in $\R^2$ supported in $D=[-0.7,0.7]^2$.
First, $q_{\mathrm{c}}= 0.4 \, \1_{D}$ is piecewise constant, second
\[
  q_{\mathrm{v}}(x) = \frac{2}{5}\, \1_{D}(x)\, \left| \min\left[ \min(x_1-0.7,-x_1)-0.7, \min(x_2-0.7,-x_2-0.7) \right] \right|
\]
for $x\in \R^2$, and third
\[
  q_{\mathrm{r}}(x) = \frac{2}{5}\, \1_{D}(x)\, \min \left[ \min( x_1-0.7,-x_1-0.7), \min( x_2-0.7,-x_2-0.7) \right] + 1
\]
for $x\in \R^2$, see Figure~\ref{fig:1}.
\begin{figure}[t!!!!!h!!!b!!]
  \centering
\begin{tabular}{c c c}
  \includegraphics[width=0.3\linewidth]{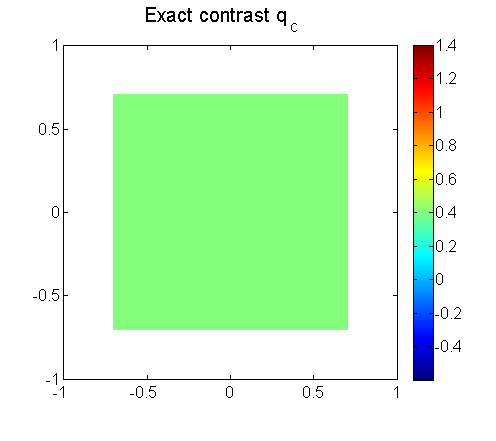}
& \includegraphics[width=0.3\linewidth]{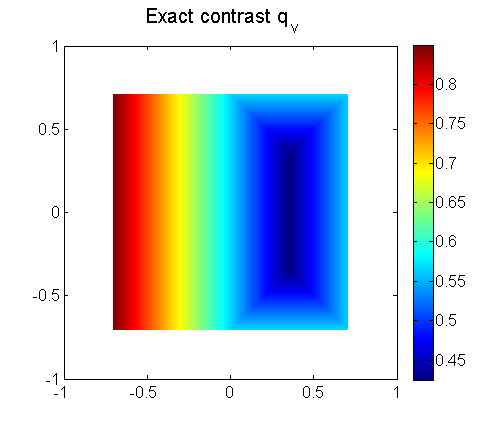}
& \includegraphics[width=0.3\linewidth]{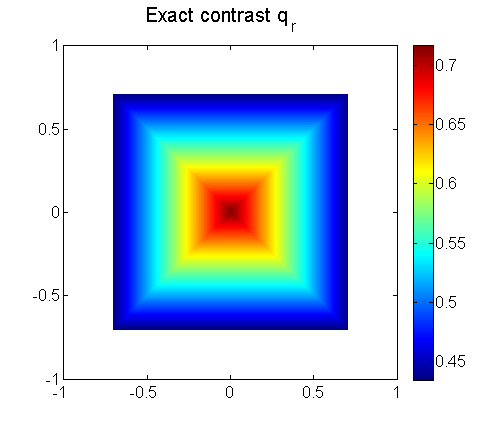}\\
(a) & (b) & (c)
\end{tabular}
\caption{(a) The contrast $q_{\mathrm{c}}$.
(b) The contrast $q_{\mathrm{v}}$.
(c) The contrast $q_{\mathrm{r}}$.}
\label{fig:1}
\end{figure}
For wave number $k=2\pi$, i.e., for wave length equal to one, the corresponding far field operators are $F_{\mathrm{c},\mathrm{v},\mathrm{r}}$.
We compare a numerical approximation of this far field operator for 32 equidistributed directions on the unit circle with numerically simulated far field operators for contrast $c \, \1_{D}$ where $c=-0.4, -0.3, \dots, 1.5$, i.e. $h=0.1$.
The simulated far field operators rely on far field data for 32 uniformly distributed incident directions computed by the spectral collocation method described in~\cite{Bur2016} (we used $2^{18}$ uniformly spaced discretization points in the domain $[-2,2]^2$).
The relative error of these synthetic far field operators is less than $10^{-4}$.
Computing one far field operator takes about 10 seconds on a Linux workstation with 4 cores and 16 GB RAM); if the support of the contrast is known in advance, one can pre-compute these auxiliary far field data.
Note that we do not add artificial noise to the simulated far field patterns, such that our numerical experiments do not allow for any statement on stability of the investigated technique.

A somewhat tricky problem for implementing the algorithm from Listing~\ref{algo} is to numerically check from a finite-dimensional approximation of $\mathcal{S}_{c \, \1_{D}}^\ast (F_{\mathrm{c}, \mathrm{v}} - F_{c\, \1_{D}})$ whether its eigenvalues tend to zero from the left (right) such that merely finitely many have a real part greater (less) than zero.
To this end, we compute first all eigenvalues in the annulus $R = \{ z\in \C: \, 10^{-8} \leq |z| \leq 10^{-2} \}$ and next the numbers $M_\pm(c)$ of eigenvalues in $R$ with real part greater (+) and less (-) than 0.
If $M_+(c)$ ($M_-(c)$) vanishes, we conclude that the eigenvalues of $\mathcal{S}_{c \, \1_{D}}^\ast (F_{\mathrm{c}, \mathrm{v}} - F_{c\, \1_{D}})$ cannot tend to zero from the right (left).
As the most expensive part of the algorithm hence is the computation of eigenvalues and eigenvectors of several matrices of size $32 \times 32$, the runtime of the presented implementation is negligible once the far field operators for the test contrasts are pre-computed.

Figure~\ref{fig:2}(a) shows plots of $M_\pm(c)$ for $c=-0.4, \dots, 1.5$ and $F=F_{\mathrm{c}}$ in (a) and $F=F_{\mathrm{v}}$ in (b).
For $q_\mathrm{c}$, $M_+(c)$ vanishes up to $c = 0.4$, whereas $M_-(c)$ vanishes for $c \geq 0.4$, such that the interior trace of the exact contrast on the boundary of the square $D$ must equal 0.4, which equals the true value.
For the spatially varying contrast $q_\mathrm{v}$, the numbers $M_+(c)$ also vanishes up to $c = 0.4$ and $M_-(c)$ vanishes for $c \geq 0.9$, such that $q_\mathrm{v}|_{\partial D}$ must take values in between $0.4$ and $0.9$.
Whilst this conclusion is true and the upper value equals the maximum of the trace $q_{\mathrm{v}}|_{\partial D}$, the lower value is about 0.15 below the minimum of that trace (and even about 0.25 below the minimum of $q_{\mathrm{v}}$ of about $0.425$.
Finally, Figure~\ref{fig:2}(c) shows that the boundary values $q_{\mathrm{r}}|_{\partial D}$ must lie in between 0.4 and 0.5, which are the best possible bounds for the chosen values of $c=-0.4, \dots, 1.5$ and the exact boundary values $q_{\mathrm{r}}|_{\partial D}=0.45$.
Note that $q_{\mathrm{r}}$ takes values in between 0.45 and 0.75, such that our theoretical results are confirmed:
Merely the boundary values of $q$ influence whether the eigenvalues of $\mathcal{S}_{c \, \1_{D}}^\ast (F_{\mathrm{c}, \mathrm{v}} - F_{c\, \1_{D}})$ tend to zero from the left or the right.
To conclude, the presented implementation indicates correct bounds for the boundary values of the contrast if the support of the exact contrast is known.

\begin{figure}[t!!!!!h!!!b!!]
  \centering
\begin{tabular}{c c c}
  \includegraphics[width=0.30\linewidth]{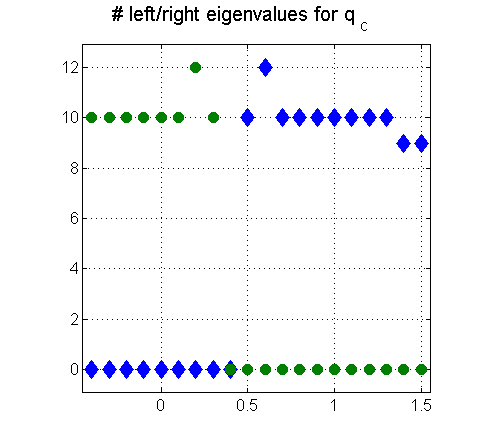}
& \includegraphics[width=0.30\linewidth]{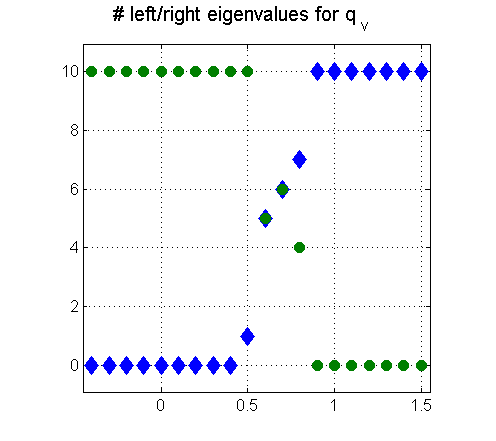}
& \includegraphics[width=0.30\linewidth]{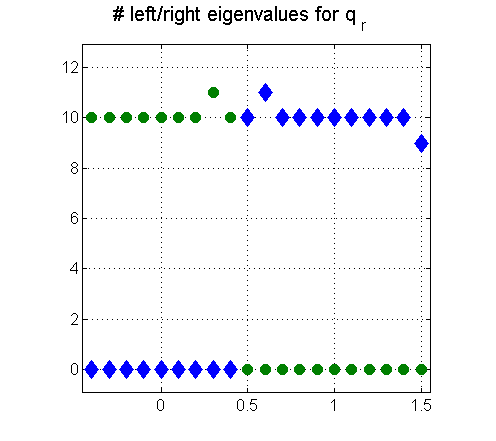}\\
(a) & (b) & (c)
\end{tabular}
\caption{Numbers of eigenvalues $M_\pm(c)$ of $\mathcal{S}_2^\ast(F-F_{c \1_{D}})$ in $\{ 10^{-8} < |z| < 10^{-2} \}$ for $c = -0.4, \dots, 1.5$ and $D=[-0.70, 0.70]^2$ with real part larger (dots, $M_+$) and smaller (diamonds $M_-$) than zero.
(a) $F=F_{q_{\mathrm{c}}}$.
(b) $F=F_{q_{\mathrm{v}}}$.
(c) $F=F_{q_{\mathrm{r}}}$.}
\label{fig:2}
\end{figure}

For more accurate space-dependent upper and lower bounds for the boundary values of a contrast function $q$, a natural idea is to replace the constant test contrasts $c\, \1_{D}$ by real-valued linear functions $p$ multiplied by the indicator function of $D$.
Initializing upper and lower approximations $q^{(\pm)}$ by constant values times $\1_{D}$ such that $q^{(-)} \leq q  \leq q^{(+)}$ in $\overline{D}$ allows to compute such bounds by checking as in Listing~\ref{algo} whether the eigenvalues of $\mathcal{S}_{p\, \1_{D}}^\ast (F_q - F_{p\, \1_{D}})$ tend to zero from the left or from the right.
(Numerically, we check as above whether the number of eigenvalues of a discretization of the latter operator of dimension $32 \times 32$ in $R_\pm = \{ z\in \C : \, 10^{-8} |z| \leq 10^{-2}, \, \Re (z) \gtrless 0 \}$ vanishes.)
If zero is limit from the left (or from the right), we conclude that $p \geq q$ (or that $p \leq q$) and update $q^{(+)}$ by $\min(p, q^{(+)})$ (and $q^{(-)}$ by $\max(p, q^{(-)})$).

As linear functions possess three degrees of freedom, the computational work of (pre-)computing far field operators to assemble discretizations of the normal operators $\mathcal{S}_{p\, \1_{D}}^\ast (F_q - F_{p\, \1_{D}})$ increases drastically compared to the algorithm from Listing~\ref{algo}.
For the examples below, we parametrized linear functions via 12 equidistributed points $x_1,\dots,x_{12}$ on the boundary of $D$ with associated directions $\hat{x}_j = x_j/|x_j|$, eleven different slopes $s_\ell = -2,-1.8, -1.6, \dots, 2$, and eleven different off-sets $o_m = 0,0.1,\dots, 1$, and approximated 1452 far field operators for contrasts $p \, \1_D$ with linear functions
\begin{equation}\label{eq:aux808}
  p(x) = s_\ell \ \hat{x}_j \cdot (x-x_j) + o_m, \qquad j=1,\dots,12, \ \ell,m=1,\dots,11.
\end{equation}
Note that again that these far field data can be pre-computed if the shape of the scattering object is known a-priori.
More generally, we could also consider polynomials of higher degree, but the amount of work to precompute far field operators increases exponentially in the degree.

\begin{figure}[h!!!!!!!!!!!!!!!!!!!!!!!!!!!!!!!!!!!!t!!!!!b!!]
  \centering
\begin{tabular}{@{\hspace*{-5mm}}c@{\hspace*{-4mm}}c@{\hspace*{-4mm}}c@{\hspace*{-4mm}}c}
\includegraphics[width=0.35\linewidth]{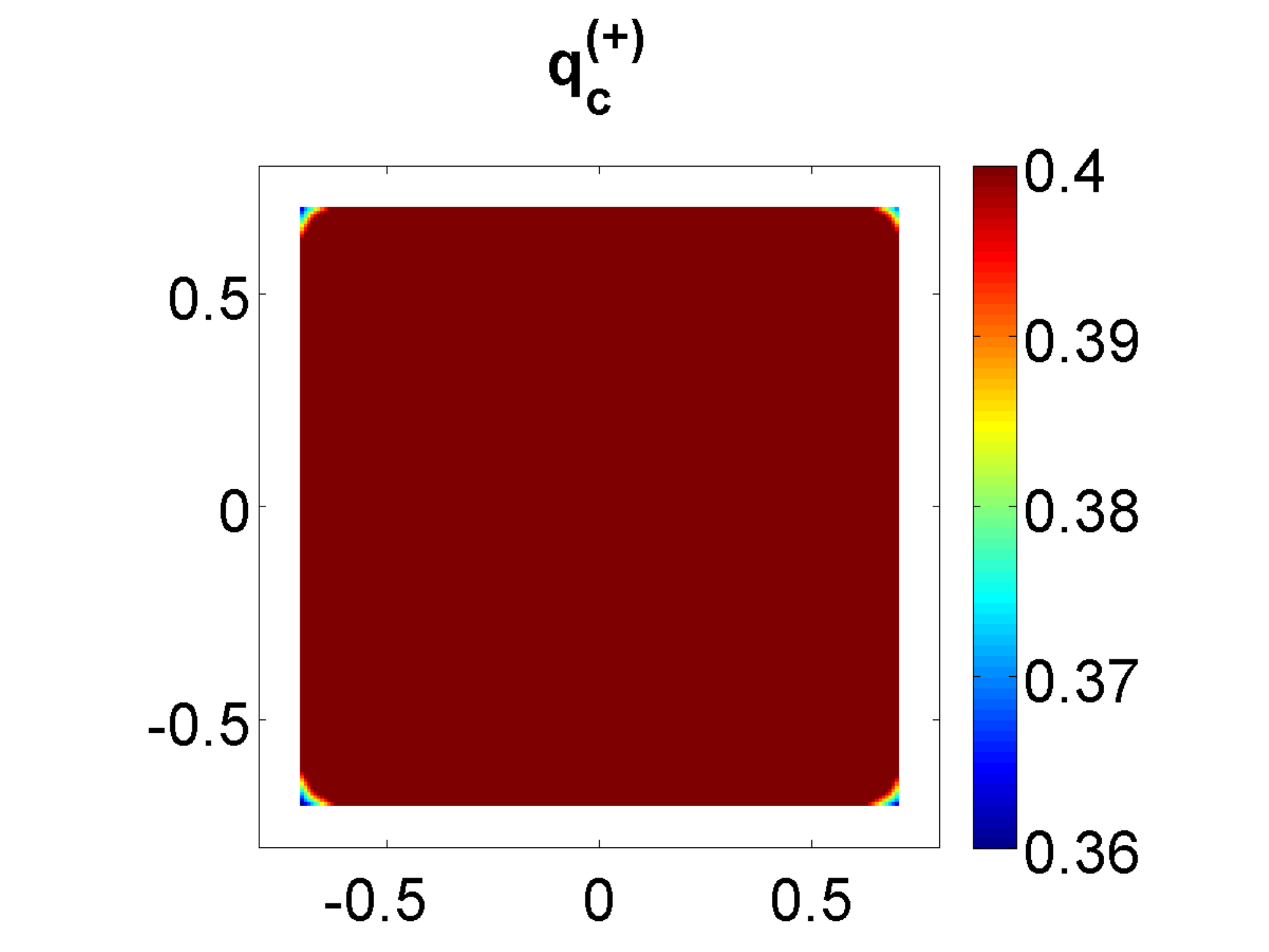}
& \includegraphics[width=0.35\linewidth]{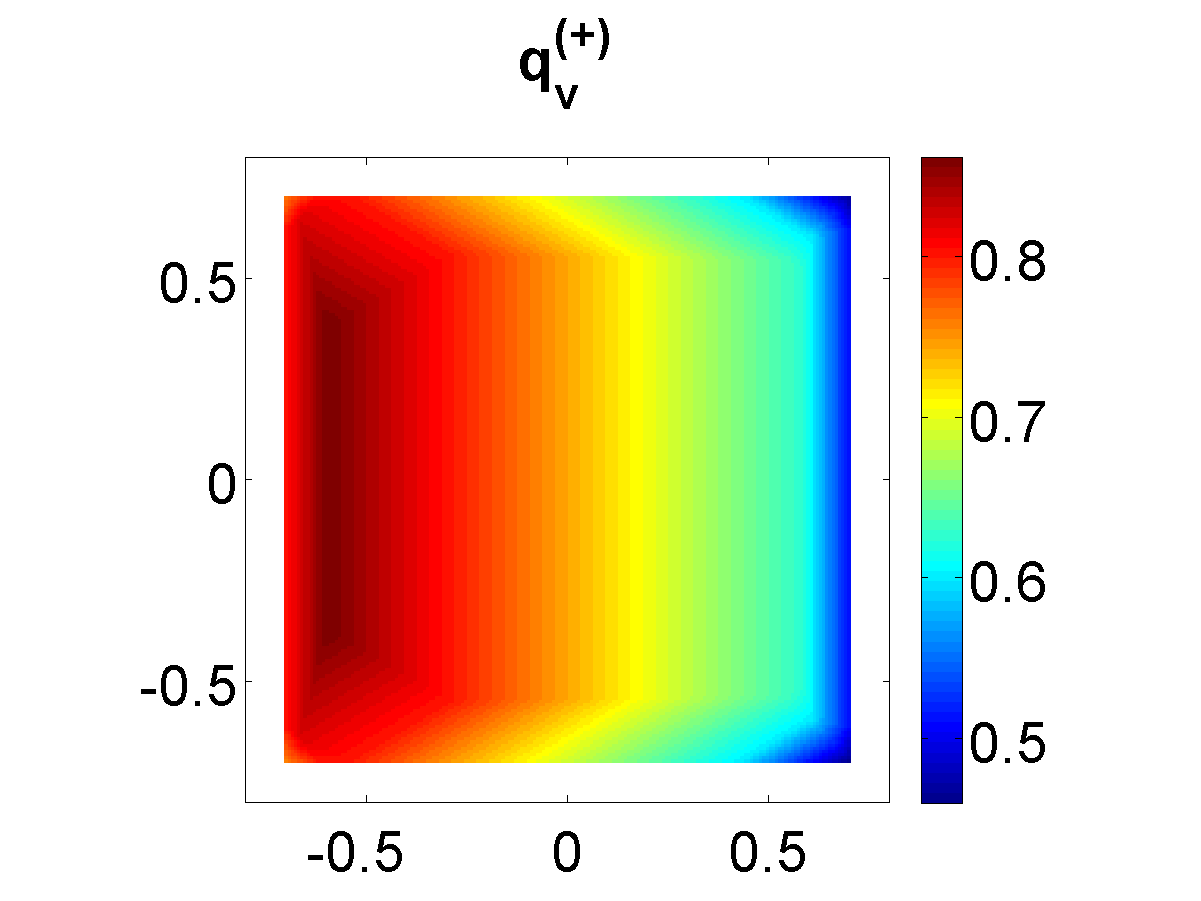}
& \includegraphics[width=0.35\linewidth]{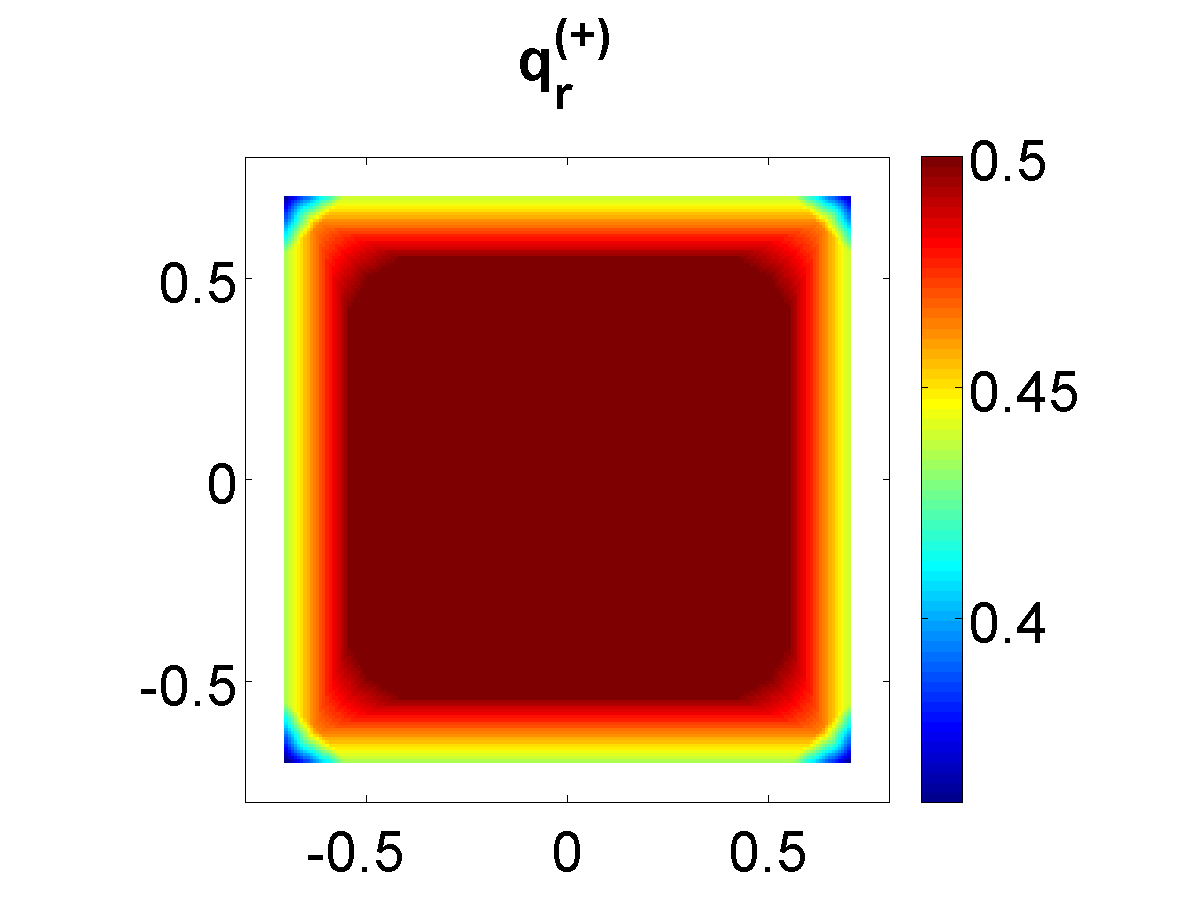}\\[-2mm]
(a) & (b) & (c)\\[0mm]
\includegraphics[width=0.35\linewidth]{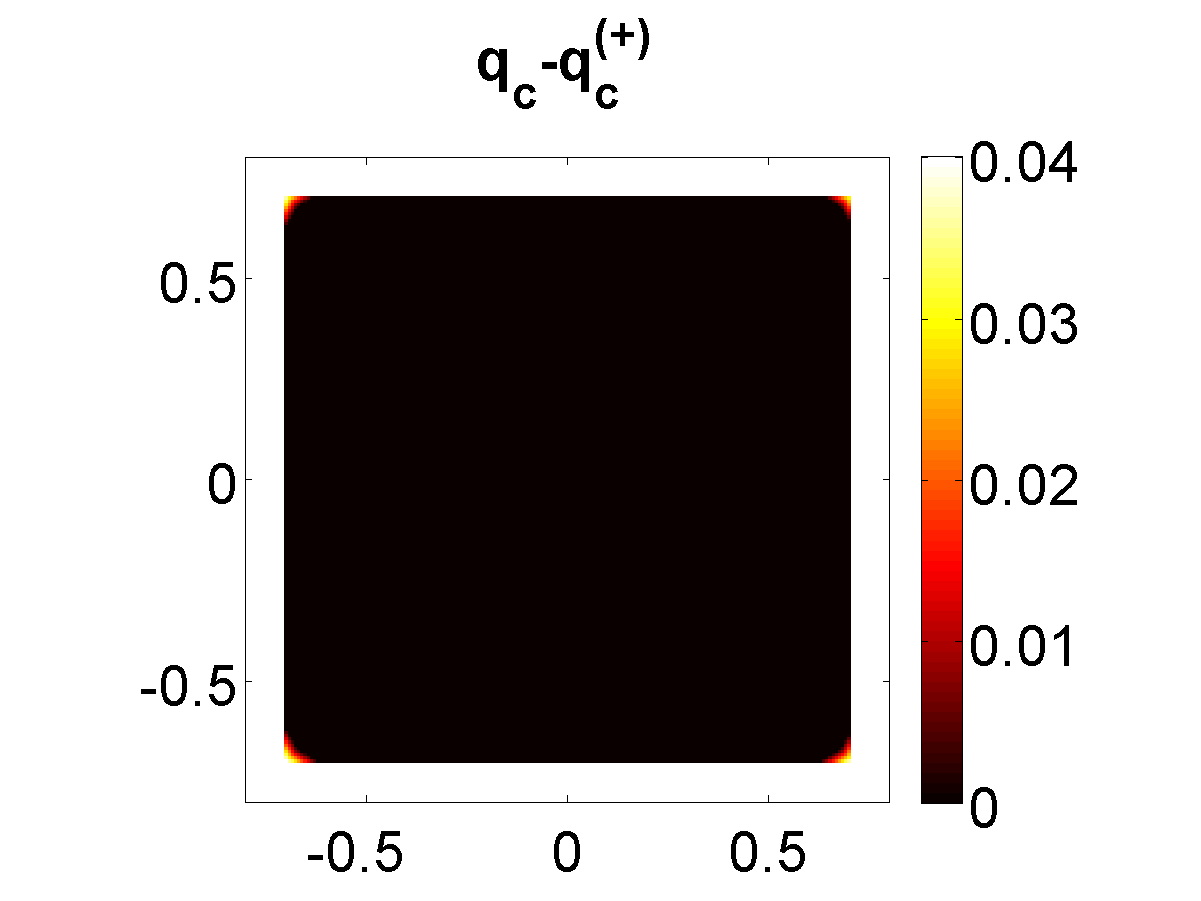}
& \includegraphics[width=0.35\linewidth]{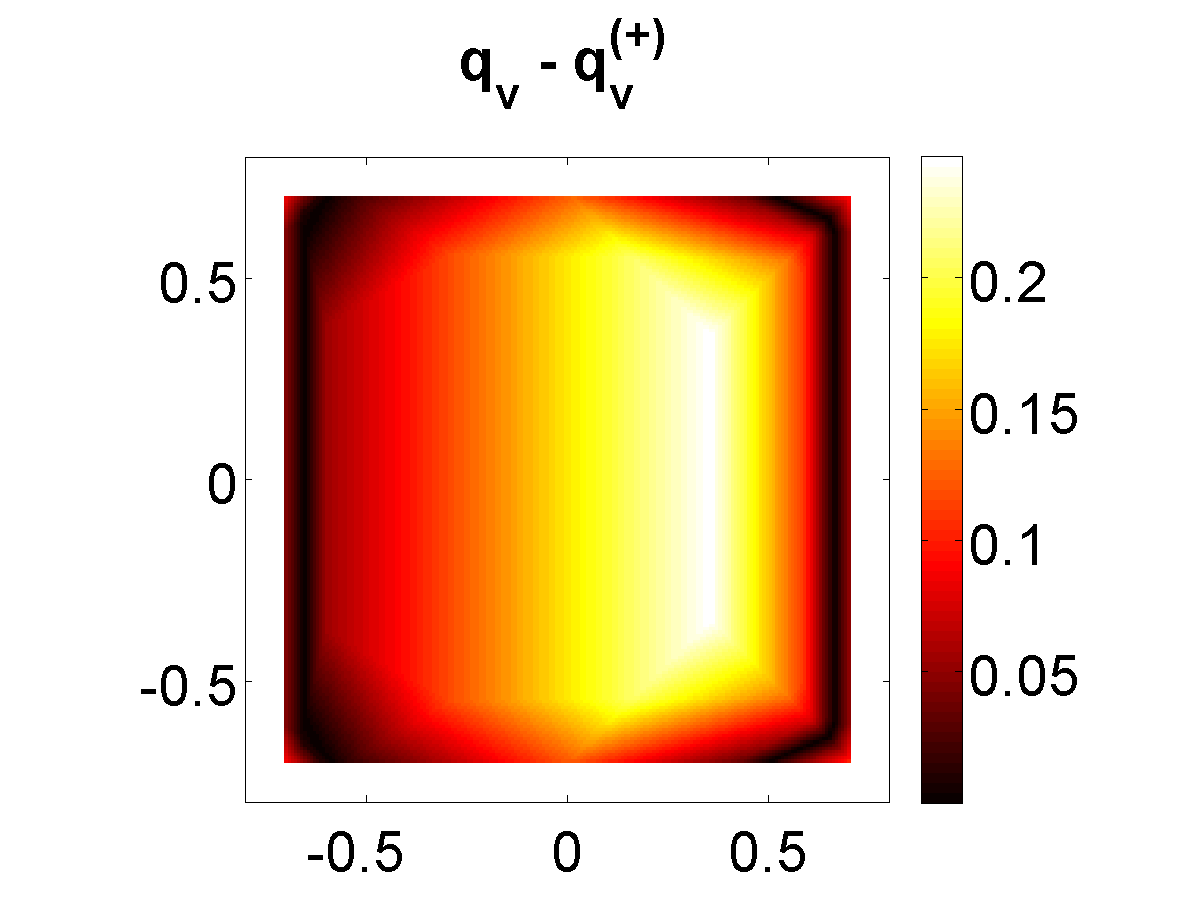}
& \includegraphics[width=0.35\linewidth]{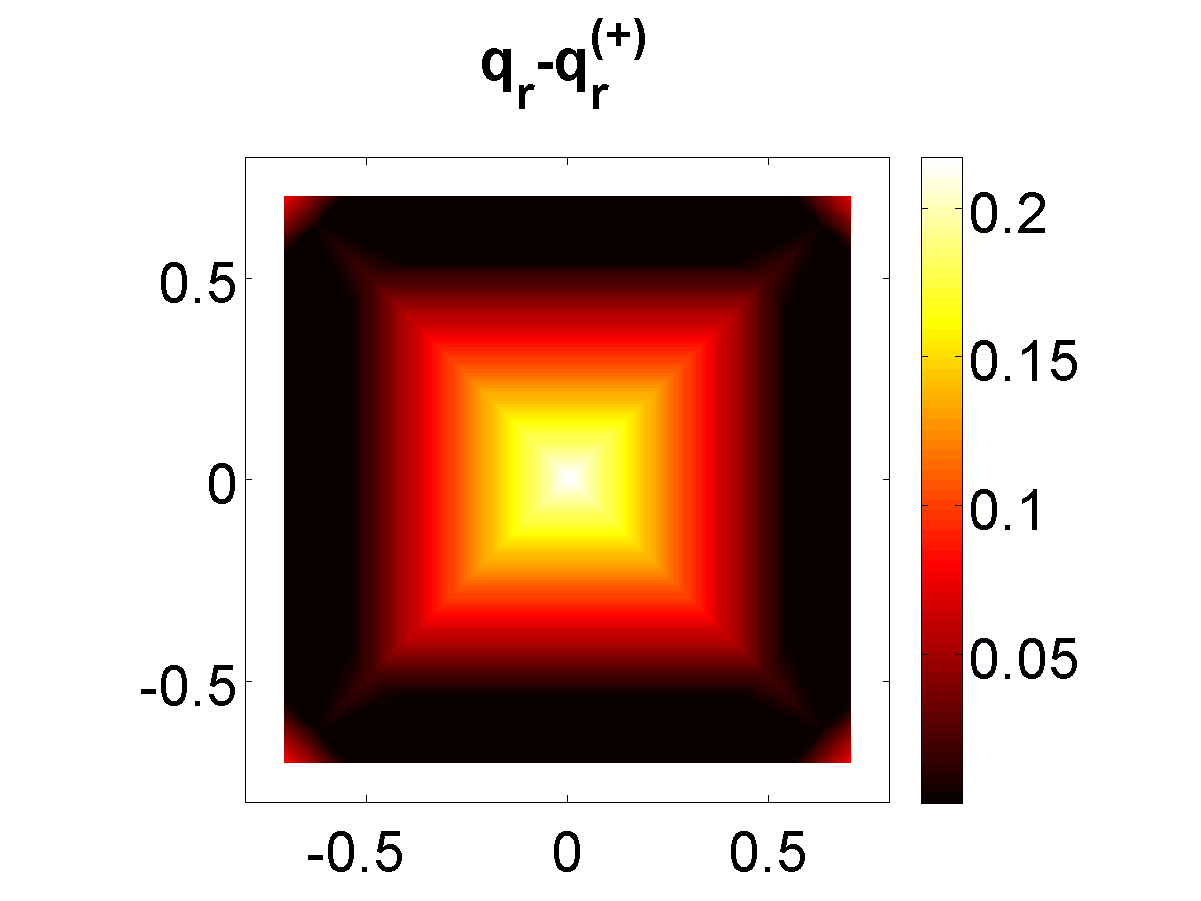}\\[-2mm]
(d) & (e) & (f)\\[0mm]
\includegraphics[width=0.35\linewidth]{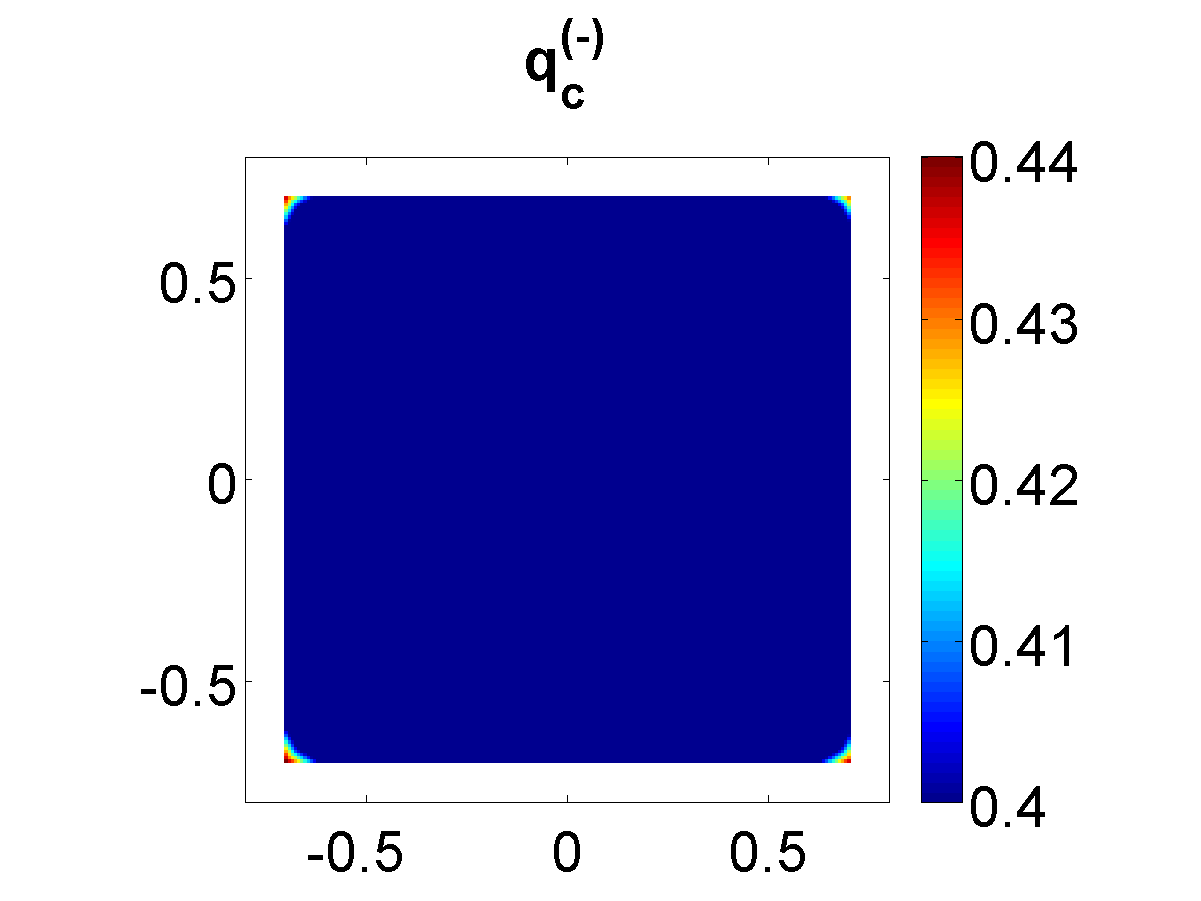}
& \includegraphics[width=0.35\linewidth]{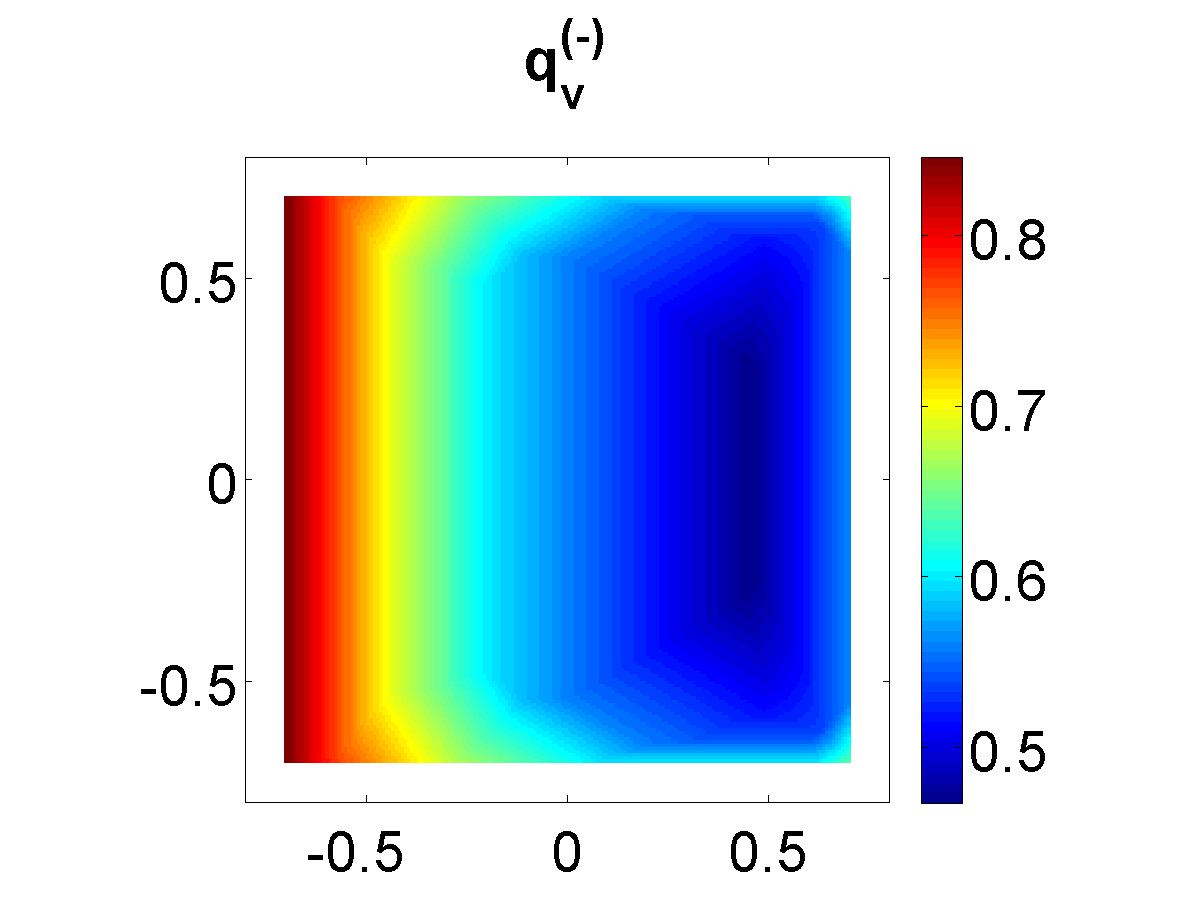}
& \includegraphics[width=0.35\linewidth]{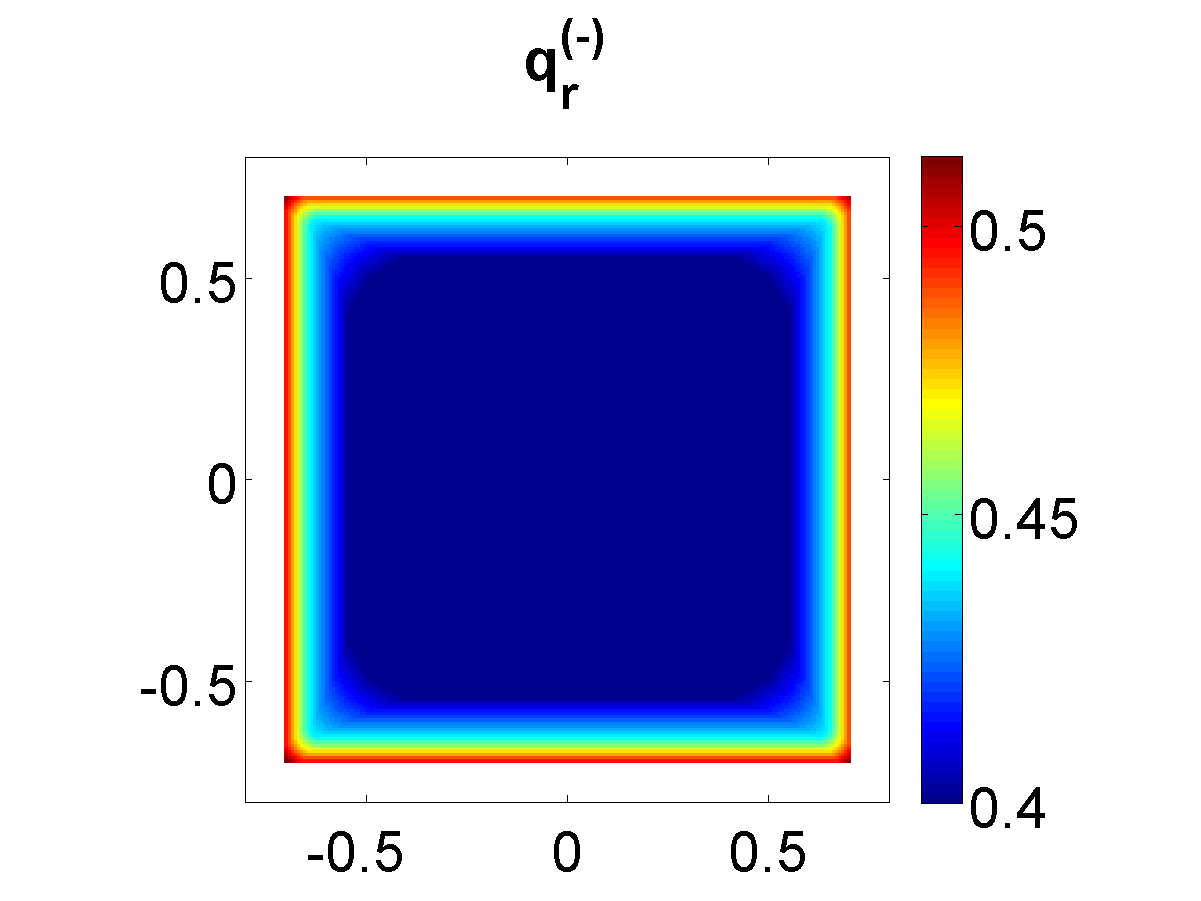}\\[-2mm]
(g) & (h) & (i)\\[0mm] 
\includegraphics[width=0.35\linewidth]{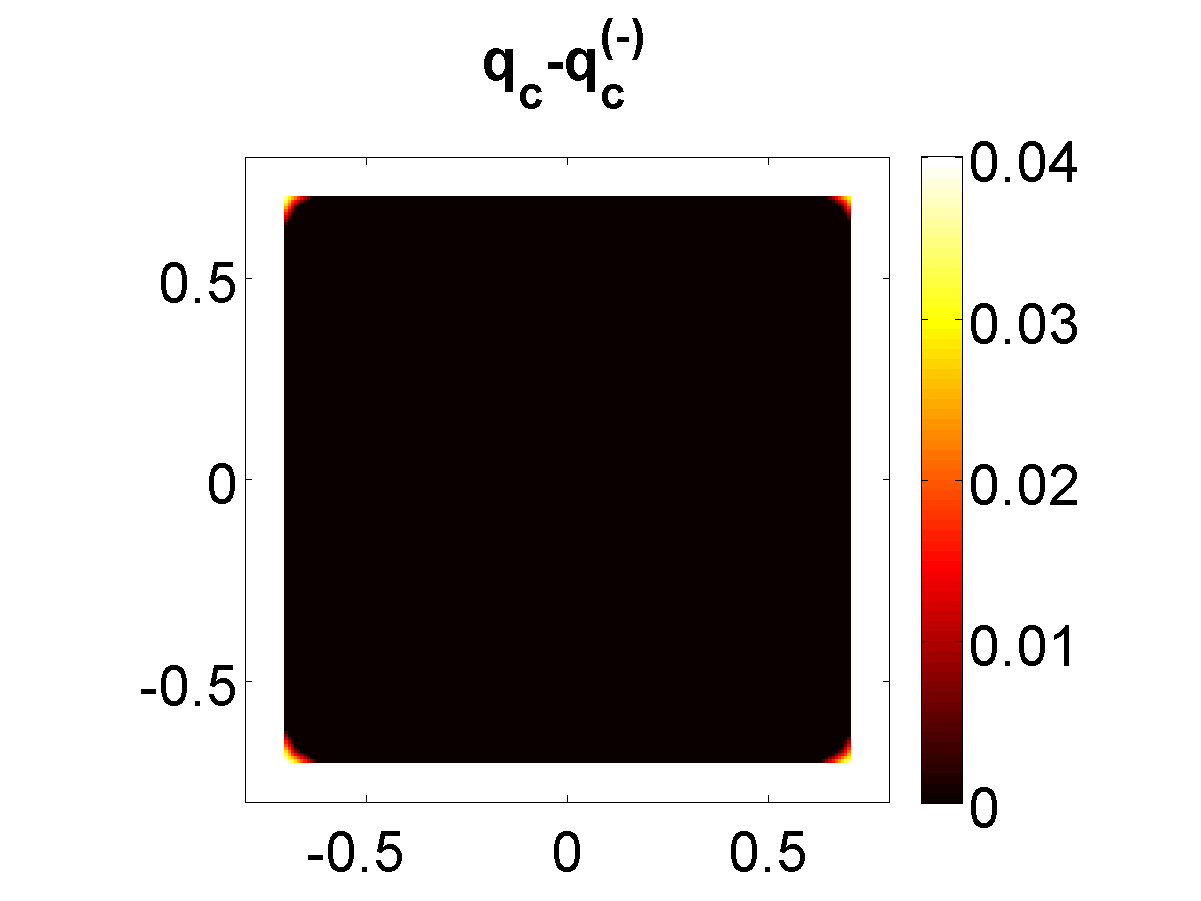}
& \includegraphics[width=0.35\linewidth]{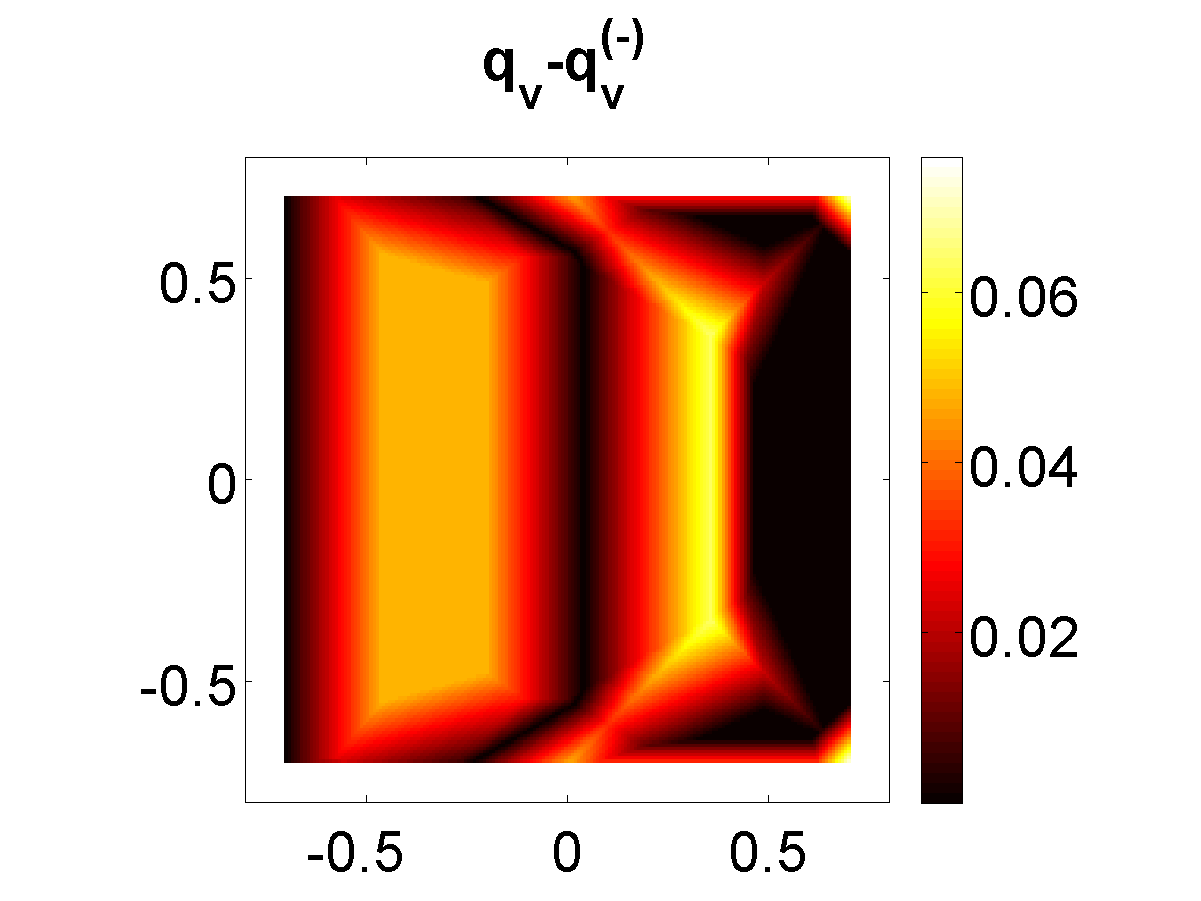}
& \includegraphics[width=0.35\linewidth]{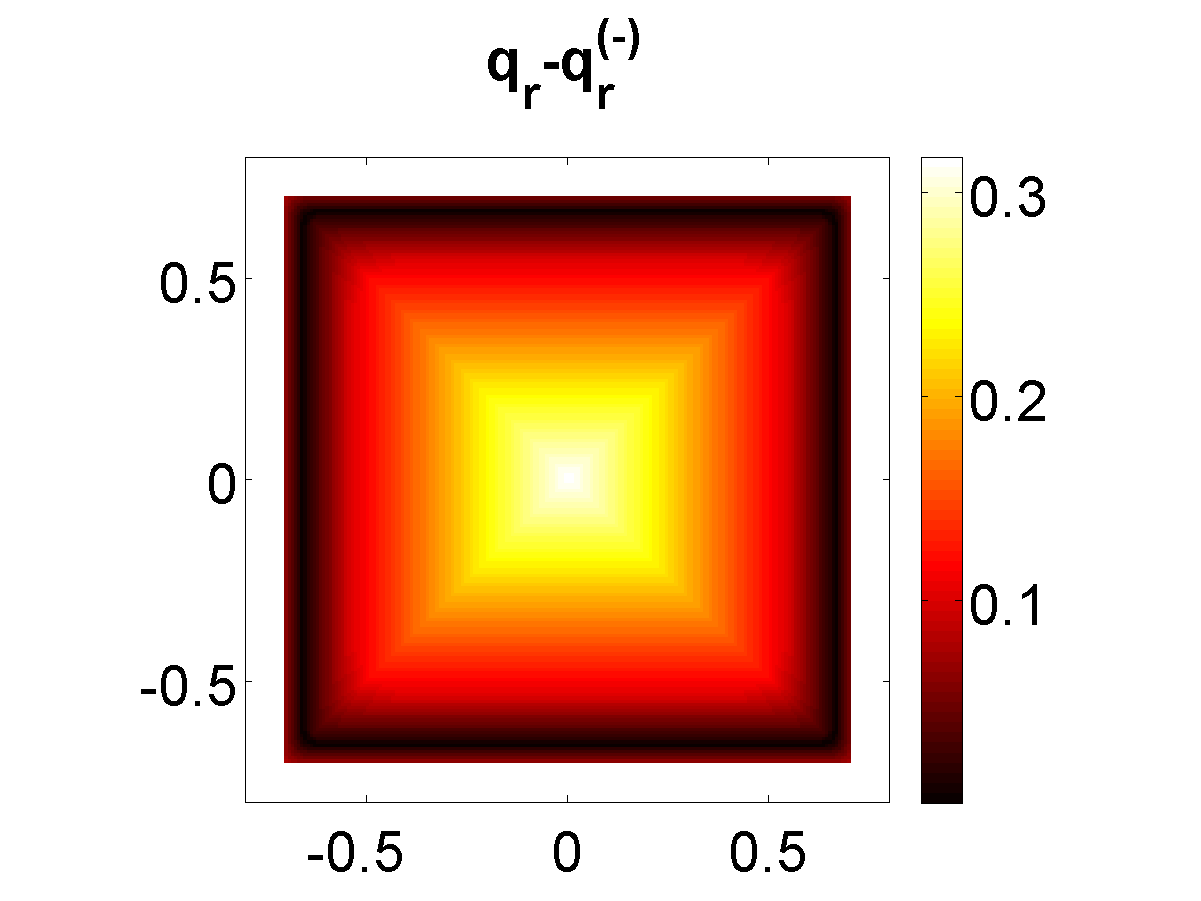}\\[-2mm]
(j) & (k) & (l)\\[-3mm]
\end{tabular}
\caption{The upper and lower bounds $q^{(\pm)}_{\mathrm{c,v,r}} \, \1_D$ computed for exact contrasts $q_{\mathrm{c,v,r}} \, \1_D$, see Figure~\ref{fig:1}, and linear comparison contrasts determined in~\eqref{eq:aux808}.
In each column, from top to bottom: $q_{\mathrm{c,v,r}}^{(+)} \, \1_D$, $(q_{\mathrm{c,v,r}}-q_{\mathrm{c,v,r}}^{(+)}) \, \1_D$, $q_{\mathrm{c,v,r}}^{(-)} \, \1_D$, and $(q_{\mathrm{c,v,r}}-q_{\mathrm{c,v,r}}^{(-)}) \, \1_D$.
First/second/third column: results for $q_{\mathrm{c}}$/$q_{\mathrm{v}}$/$q_{\mathrm{r}}$.}
\label{fig:3}
\end{figure}

Figure~\ref{fig:3} shows the resulting approximations $q_{\mathrm{c,v,r}}^{(\pm)} \, \1_D$ for the three exact contrasts $q_{\mathrm{c,v,r}}$ shown in Figure~\eqref{fig:1}.
(We initialized $q^{(\pm)}$ as $\pm 10^3 \, \1_{D}$.)
Whilst the maximal norm $\| q_{\mathrm{c}} - q_{\mathrm{c}}^{(\pm)}\|_{L^\infty(\partial D)}$ is about 0.04, $\| q_{\mathrm{r}} - q_{\mathrm{r}}^{(\pm)}\|_{L^\infty(\partial D)}$ is about 0.07; $\| q_{\mathrm{r}} - q_{\mathrm{r}}^{(+)}\|_{L^\infty(\partial D)}$ is about 0.1 and $\| q_{\mathrm{r}} - q_{\mathrm{r}}^{(-)}\|_{L^\infty(\partial D)}$ about 0.07.
This shows that the boundary values of $q_{\mathrm{c,v,r}}$ are well-approximated by their piecewise linear bounds.
The extrema of the above-mentioned differences maxima are always attained in one of the four corners, which, arguably, is natural as theory requires smooth domains.
Clearly, both bounds do not approximate the exact contrasts inside the domain $D$ unless that exact contrast is constant in $D$.
Since we deal with linear test contrasts, the upper and lower bounds $q^{(\pm)}$ are however concave and convex, respectively, as pointwise minimum and maximum over linear functions (see, e.g., Figure~\ref{fig:3}(e) and (g)).
Thus, approximating boundary values that fail to be either concave of convex certainly requires quadratic comparison functions to obtain a comparable accuracy.


\providecommand{\noopsort}[1]{}

\end{document}